\documentclass[10pt]{amsart}

\usepackage{amsfonts}
\usepackage{amssymb}
\usepackage{color}
\usepackage{amsmath, amsthm, latexsym}
\usepackage[all]{xy}
\usepackage{hyperref}
\usepackage[margin=1.5in]{geometry}
\usepackage{graphicx}

\def \deg{\textup{deg}}
\def \tr{\textup{tr}}
\def \dim{\textup{dim}}
\def \rank{\textup{rank}}
\def \ord{\textup{ord}} 
\def \conv{\textup{conv}}
\def \vol{\textup{vol}}

\def \In{\textup{in}}

\def \span{\textup{span}}

\def \End{\textup{End}}

\theoremstyle{plain}
\newtheorem{theorem}{Theorem}[section]
\newtheorem{lemma}[theorem]{Lemma}
\newtheorem{corollary}[theorem]{Corollary}
\newtheorem{proposition}[theorem]{Proposition}

\newtheorem{THM}{Theorem}
\newtheorem{LEM}[THM]{Lemma}
\newtheorem{CORO}[THM]{Corollary}
\newtheorem{PROB}{Problem}

\theoremstyle{definition}
\newtheorem{definition}[theorem]{Definition}

\newtheorem{remark}[theorem]{Remark}
\newtheorem{example}[theorem]{Example}
\newtheorem{algorithm}[theorem]{Algorithm}

\newtheorem{DEF}{Definition}

\newtheorem*{REM}{Remark}

\numberwithin{equation}{section}

\def\T{\mathcal{T}}
\def\V{\mathcal{V}}
\def\gr{\mathrm{gr}}
\def\Gr{\mathrm{Gr}}
\def\GL{\mathrm{GL}}

\def\SL{\mathrm{SL}}

\def \B{\mathcal{B}}
\def \A{\mathcal{A}}
\def \F{\mathcal{F}}

\def\C{\mathbb{C}}
\def\R{\mathbb{R}}
\def\Z{\mathbb{Z}}
\def\N{\mathbb{N}}

\def\Q{\mathbb{Q}}

\def \Gm{\mathbb{G}_\textup{m}}

\def \k{{\bf k}}
\def \x{{\bf x}}
\def \u{u}
\def \v{\mathfrak{v}}
\def \F{\mathcal{F}}

\def \w{M}
\def \MIN{\textup{MIN}}
\def \MAX{\textup{MAX}}
\def \an{\textup{an}}
\def \S{\underline{w}_0}
\def \GR{\textup{GR}}

\def\q{/\!/}
\def\ql{\backslash \! \backslash}

\def\Proj{\mathrm{Proj} \,}
\def \Spec{\mathrm{Spec} \,}

\begin{document}
\title[Khovanskii bases, higher rank valuations and tropical geometry]{Khovanskii bases, higher rank valuations and tropical geometry}
\author{Kiumars Kaveh}
\address{Department of Mathematics, University of Pittsburgh,
Pittsburgh, PA, USA}
\email{kaveh@pitt.edu}

\author{Christopher Manon}
\address{Department of Mathematics, University of Kentucky, Lexington, KY, USA}
\email{Christopher.Manon@uky.edu}

\date{\today}

\thanks{The first author is partially supported by a National Science Foundation Grant 
(Grant ID: DMS-1601303), Simons Foundation Collaboration Grant for Mathematicians, and Simons Fellowship.}

\thanks{The second author is partially supported by a National Science Foundation Grant (Grant ID: DMS-1500966).}


\begin{abstract}
Given a finitely generated algebra $A$, it is a fundamental question whether $A$ has a full rank discrete (Krull) valuation $\v$ with finitely generated value semigroup. We give a necessary and sufficient condition for this, in terms of tropical geometry of $A$. In the course of this we introduce the notion of a Khovanskii basis for $(A, \v)$ which provides a framework for far extending Gr\"obner theory on polynomial algebras to general finitely generated algebras.
In particular, this makes a direct connection between the theory of Newton-Okounkov bodies and tropical geometry, and toric degenerations arising in both contexts. We also construct an associated compactification of $\Spec(A)$. Our approach includes many familiar examples such as the Gel'fand-Zetlin degenerations of coordinate rings of flag varieties as well as wonderful compactifications of reductive groups. We expect that many examples coming from cluster algebras naturally fit into our framework.
\end{abstract}

\maketitle

\setcounter{tocdepth}{1}
\tableofcontents


\section{Introduction}
It is an important question in commutative algebra and algebraic geometry whether a given finitely generated algebra has a full rank valuation with finitely generated value semigroup. The purpose of this paper is to give a necessary and sufficient condition for this in terms of tropical geometry. In the course of this, we introduce the notion of a Khovanskii basis. {In this terminology, the main results of the paper concern necessary and sufficient conditions for existence of a finite Khovanskii basis.}  


The theory of Khovanskii bases, developed in this paper, opens doors to extend powerful and extremely useful methods of Gr\"obner basis theory for ideals in a polynomial ring to general algebras. Whenever an algebra $A$ (equipped with a full rank valuation $\v$) has a Khovanskii basis, one can do many computations in the algebra algorithmically. 
This should enable an extension of Gr\"{o}bner basis theory to ideals in algebras with a finite Khovanskii basis.
In particular, extending the SAGBI basis polyhedral homotopy method, algorithms can be developed to find solutions of systems of equations from the algebra $A$ (in the sense of \cite{KKh-Divisors, KKh-Equations}). This has direct applications to problems from applied algebraic geometry such as finding solutions of non-sparse systems of polynomial equations appearing in robotics or chemical reaction networks. {Some of these applications were discussed in the mini-symposium {\it Newton-Okounkov bodies and Khovanskii bases} in SIAM Conference in Applied Algebraic Geometry (Atlanta, 2017). As far as the authors know, the idea of homotopy methods for computing solutions of systems of equations in the context of Khovanskii bases is due to Bernd Sturmfels (see also \cite{HSS}).} One can think of the theory of Khovanskii bases as the computational and algorithmic side of the general theory of Newton-Okounkov bodies. As Example \ref{ex-Gobel-revisited} indicates the scope of Khovanskii basis theory is far larger than SAGBI theory.

We would like to mention that the theory of Khovanskii bases has been used in the study of Gaussoids \cite{Gaussoid}. Also recently Duff and Sottile use Khovanskii bases in a method to certify solutions to systems of polynomial equations \cite{Duff-Sottile}.   

{Before stating the main results of the paper, let us review some background material.} Let $A$ be a finitely generated $\k$-algebra and domain with Krull dimension $d$ over a field $\k$. We consider a discrete valuation.  $\v: A \setminus \{0\} \to \Q^r$, for some $0 < r \leq d$, which lifts the trivial valuation on $\k$ (here the additive group $\Q^r$ is equipped with a group ordering $\succ$, see Definition \ref{def-valuation}). The image $S(A, \v)$ of $\v$, that is, $$S(A, \v) = \{ \v(f) \mid 0 \neq f \in A \},$$ is usually called the {\it value semigroup of $\v$}. It is a (discrete) additive semigroup in $\Q^r$. The {\it rank} of the valuation $\v$ is the rank of the group generated by its value semigroup.
  The valuation $\v$ gives a filtration $\mathcal{F}_\v = (F_{\v \succeq a})_{a \in \Q^r}$ on $A$, defined by:$$F_{\v \succeq a} = \{ f \in A \mid \v(f) \succeq a\} \cup \{0\}.$$
($F_{\v \succ a}$ is defined similarly.) The corresponding associated graded $\gr_\v(A)$ is:
$$\gr_\v(A) = \bigoplus_{a \in \Q^r} F_{\v \succeq a} / F_{\v \succ a}.$$
It is important to note that $\gr_\v(A)$ is also a domain. 
For $0 \neq f \in A$ we can consider its image $\bar{f}$ in $\gr_\v(A)$, namely the image of $f$ in $F_{\v \succeq a} / F_{\v \succ a}$ where $a = \v(f)$. The following is a central concept in the paper.

\begin{DEF}[Khovanskii basis] \label{def-intro-Khovanskii-basis}
A set $\B \subset A$ is a {\it Khovanskii basis} for $(A, \v)$ if the image of $\B$ in the associated graded $\gr_\v(A)$ forms a set of algebra generators.
\end{DEF}

The case when our algebra $A$ has a finite Khovanskii basis with respect to a valuation $\v$ is particularly desirable. 

\begin{REM}
The main idea behind the definition of a Khovanskii basis is to obtain information about $A$ from its associated graded algebra $\gr_\v(A)$. This algebra can be regarded as a degeneration of $A$ and is often simpler to work with, for example, $\gr_\v(A)$ is graded by the value semigroup $S = S(A, \v)$. The case of main interest is when $\k$ is algebraically closed and $\v$ has full rank equal to $d$. In this case the associated graded $\gr_\v(A)$ is the semigroup algebra $\k[S]$, and hence is a subalgebra generated by monomials in a polynomial algebra. The existence of a finite Khovanskii basis then is equivalent to $S$ being a finitely generated semigroup, in which case $\k[S]$ basically can be described by combinatorial data (see Section \ref{sec-valandKhovanskii} and Proposition \ref{prop-grA-semigp-algebra}). Moreover, we have a degeneration of $\Spec(A)$ to the (not necessarily normal) toric variety $\Spec(\k[S])$.
\end{REM}


The notion of a Khovanskii basis generalizes the notion of a SAGBI basis (also called a canonical basis in \cite{St}) which is used when $A$ is a subalgebra of a polynomial algebra (see \cite{SAGBI}, \cite[Chapter 11]{St} and also Remark \ref{rem-SAGBI}). The name Khovanskii basis was suggested by B. Sturmfels in honor of A. G. Khovanskii's contributions to combinatorial algebraic geometry and convex geometry. {As far as the authors know, the present paper is the first paper which deals with the general notion of a Khovanskii basis}. In this paper, after developing some basic facts about Khovanskii bases, we give a necessary and sufficient condition for existence of  a finite Khovanskii basis for $A$. We find that tropical geometry provides a suitable language for this condition (see Theorems \ref{th-intro-main0} and \ref{th-intro-main1} below). 

There is a simple classical algorithm to represent every element in $A$
as a polynomial in elements of a Khovanskii basis $\B$. This is usually known as the {\it subduction algorithm} (Algorithm \ref{algo-subduction}). In general, given $(A, \v)$ and $f \in A$, it is possible that the subduction algorithm does not terminate in finite time (see Example \ref{ex-subduction-not-terminate}).
We will be interested in the cases where it does terminate. It is easily seen that this happens if the value semigroup $S(A, \v)$ is maximum well-ordered, i.e., every increasing chain in $S(A, \v)$ has a maximum (Proposition \ref{prop-subduction}). This is the case if $\v$ is a homogeneous valuation with respect to a positive grading on $A$, i.e. an algebra grading by $\Z_{\geq 0}$.

Let us say few words about the important case when $A$ is positively graded. In this case, it is convenient to consider a valuation which also encodes information about the grading. More specifically one would like to work with a valuation $\v: A \setminus \{0\} \to \N \times \Q^{r-1} \subset \Q \times \Q^{r-1}$ such that the first component of $\v$ is the degree. That is, for any $0 \neq f \in A$ we have:\footnote{For this to be a valuation one should consider reverse ordering on the first coordinate. Alternatively, one can define $\v(f) = (-\deg(f), \cdot)$.}
\begin{equation}  \label{equ-intro-v-degree}
\v(f) = (\deg(f), \cdot).
\end{equation}
To $(A, \v)$, with $\v$ as in \eqref{equ-intro-v-degree}, one associates a convex body $\Delta(A, \v)$ in $\R^{r-1}$ called a Newton-Okounkov body. This convex body encodes information about the Hilbert function of the algebra $A$ (see Section \ref{subsec-NO-bodies} as well as \cite{Ok, LM, KK}). When $\k$ is algebraically closed and $\v$ has full rank, one shows that $\Delta(A, \v)$ is a convex body whose dimension is the degree of the Hilbert polynomial of $A$ and its volume is the leading coefficient of this Hilbert polynomial (\cite[Theorem 2.31]{KK}).\footnote{In fact, this statement is still true when $A$ is not necessarily finitely 
generated as an algebra, but is contained in a finitely generated graded algebra.} In particular, if $A$ is the homogeneous coordinate ring of a projective variety $Y$, then the degree of $Y$ is given by $\dim(Y)!$ times the volume of the convex body $\Delta(A, \v)$ (\cite[Corollary 3.2]{KK}). We would like to point out that the finite generation of the value semigroup $S = S(A, \v)$ implies that the corresponding Newton-Okounkov body is a rational polytope. Moreover, we have a toric degeneration of $Y = \Proj(A)$ to a (not-necessarily normal) toric variety $\Proj(\k[S])$ (see \cite{Anderson}, \cite[Section 7]{Kaveh-Crystal} and \cite{Teissier}). {The normalization of $\Proj(\k[S])$ is the toric variety associated to the polytope $\Delta(A, \v)$.}

We now explain the main results of the paper in some detail.
Let us go back to the non-graded case and as before 
let $A$ be a finitely generated $\k$-algebra and domain with Krull dimension $d$. Throughout we use the following notation and definitions: Let $\B = \{b_1, \ldots, b_n\}$ be a set of algebra generators for $A$. The set $\B$ determines a surjective homomorphism $\pi: \k[x_1, \ldots, x_n] \to A$ defined by $\pi(x_i) = b_i$, $i=1, \ldots, n$, which we refer to as a presentation of $A$.
Let $I$ be the kernel of the homomorphism $\pi$. 
Recall that the the tropical variety $\T(I)$ is the set of all $\u \in \Q^n$ such that the corresponding initial ideal $\In_\u(I)$ contains no monomials.\footnote{Conceptually it is more appropriate to talk about 
the tropical variety of an ideal in a Laurent polynomial algebra as opposed to a polynomial algebra. So in fact, instead of tropical variety of the ideal $I$ one should consider the tropical variety of the ideal generated by $I$ in the Laurent polynomial algebra $\k[x_1^\pm, \ldots, x_n^\pm]$. The tropical variety then encodes the behavior at infinity of the subscheme defined by this ideal in all possible toric completions of the ambient torus $\Gm^n$.} One knows that the tropical variety $\T(I)$ has a 
fan structure coming from the Gr\"obner fan of the homogenization of $I$ 
(\cite[Chapter 2]{MSt}). In particular, each open cone $C \subset \T(I)$ has an associated initial ideal $\In_C(I)$ (see Section \ref{sec-groebner}).\footnote{By an open cone we mean a cone that coincides with its relative interior.} 

Let $C$ be an open cone in the tropical variety $\T(I)$. We say that $C$ is a {\it prime cone} if the corresponding initial ideal $\In_C(I) \subset \k[x_1, \ldots, x_n]$
is a prime ideal.\footnote{By abuse of terminology, occasionally we may refer to a closed cone as a prime cone, in which case we mean that its relative interior is a prime cone.} The first main result of the paper is the following (Section \ref{sec-val-from-prime-cone}). 
\begin{THM} \label{th-intro-main0}
For each prime cone $C$ in the tropical variety $\T(I)$ one can construct a discrete valuation $\v=\v_C: A \setminus \{0\} \to \Q^d$ such that: (1) $\B$ is a finite Khovanskii basis for $(A, \v)$, (2) the rank of $\v$ is at least the dimension of the prime cone $C$ and, (3) the associated graded $\gr_\v(A)$ is isomorphic to $\k[x_1, \ldots, x_n]/\In_C(I)$.
\end{THM}

We call a valuation $\v: A \setminus \{0\} \to \Q^r$ a {\it subductive valuation}
if it possesses a finite Khovanskii basis and the subduction algorithm, with respect to this Khovanskii basis, always terminates (Definition \ref{def-subductive-val}).
The next theorem is the second main result of the paper. It shows that, when $A$ is positively graded, subductive valuations are exactly valuations that arise from prime cones in the tropical variety.

\begin{THM} \label{th-intro-main1}
Let $A = \bigoplus_{i \geq 0} A_i$ be a finitely generated positively graded $\k$-algebra and domain.
With notation as before, we have the following: Any valuation $\v$ with a finite Khovanskii basis consisting of homogeneous elements is subductive. Furthermore, a finite subset $\B \subset A$ consisting of homogeneous elements is a Khovanskii basis for a subductive valuation $\v$ of rank $r$ if and only if the tropical variety $\T(I)$ contains a prime cone $C$ with $\dim(C) \geq r$. 
\end{THM}

In our setting, it is also natural to introduce a generalization of the notion of a standard monomial basis from Gr\"obner theory. We call a $\k$-vector space basis $\mathbb{B}$ for $A$ an {\it adapted basis} for $(A, \v)$ if the image of $\mathbb{B}$ in the associated graded algebra $\gr_\v(A)$ forms a vector space basis for this algebra. One can perform a vector space analogue of the subduction algorithm with respect to $\mathbb{B}$ 
(Algorithm \ref{algo-vec-space-subduct}). Subductive valuations that have adapted bases are particularly nice. We see that, when $A$ is positively graded, any subductive valuation has an adapted basis.

In representation theory context, the (dual) canonical basis of Kashiwara-Lusztig provides an important example of an adapted basis for the algebra of unipotent invariants on a reductive group. Other variants of adapted bases in representation theory have been studied by Feigin, Fourier, and Littelmann in \cite{FFL}, where they are called essential bases (see Example \ref{ex-adapted-basis-rep-theory}).

Theorem \ref{th-intro-main1} is a consequence of two constructions given in Sections \ref{sec-val-from-prime-cone} and \ref{sec-cone-from-subductive-val} respectively (see Propositions \ref{prop-val-from-cone} and \ref{prop-prime-cone-from-subductive-val}). These are the technical heart of the paper. A central concept in these constructions is that of a weight valuation introduced in Section \ref{sec-weightandsubductive}. Below we briefly explain what a weight valuation is, and summarize the results of Sections \ref{sec-val-from-prime-cone} and \ref{sec-cone-from-subductive-val} in a couple of theorems. 

In the classical Gr\"obner theory, one defines the initial form $\In_\u(f)$ of a polynomial $f \in \k[x_1, \ldots, x_n]$ with respect to a weight vector $\u \in \Q^n$. This in turn gives a valuation $\tilde{\v}_\u: \k[x_1, \ldots, x_n] \setminus \{0\} \to \Q$. We use an extension of this notion and for any integer $r >0$ and an $r \times n$ matrix $\w \in \Q^{r \times n}$, we define a valuation $\tilde{\v}_\w: \k[x_1, \ldots, x_n] \setminus \{0\} \to \Q^r$ (see Section \ref{subsec-weight-quasival}). {We can then consider the pushforward of $\tilde{\v}_\w$ via the map $\pi: \k[x_1, \ldots, x_n] \to A$ to obtain a map $\v_\w: A \to \Q^r \cup \{\infty\}$ (Definition \ref{def-weight-quasival}).\footnote{Throughout the paper we will only be interested in $M$ such that $\v_M(f) \neq \infty$, for all $0 \neq f$.} In general the map $\v_\w$ is only a quasivaluation \footnote{Recall that a quasivaluation $\v$ is defined with the same axioms as a valuation except that $\v(fg) \succeq \v(f) + \v(g)$. Some authors use the term semivaluation instead of quasivaluation.} (see Section \ref{subsec-quasival-filtration}). We call a quasivaluation of the form $\v_\w$ a {\it weight quasivaluation}.} We remark that when the associated initial ideal $\In_\w(I)$ is prime then $\v_\w$ is indeed a valuation (Lemma \ref{lem-gr_v-in_w}). The following key statement in the paper relates the notions of a subductive valuation and a weight valuation (Section \ref{lem-subductiveweight}, see also Theorem \ref{th-Khovanskii-basis-equiv-conditions}): 
\begin{LEM}   \label{lem-intro-subductive-weight-val}
Let $\v: A \setminus \{0\} \to \Q^r$ be a valuation and let $\B = \{b_1, \ldots, b_n\}$ be an algebra generating set.  Then the statements:
\begin{enumerate}
\item[(1)] $\v$ is a subductive valuation with respect to $\B \subset A$,
\item[(2)] $\v$ has an adapted basis $\mathbb{B}$ consisting of monomials in $\B$,
\item[(3)] $\v$ coincides with the weight valuation $\v_M$ for the matrix $M \in \Q^{r\times n}$ with column vectors $\v(b_1), \ldots, \v(b_n)$,
\item[(4)] $\v$ has Khovanskii basis $\B$, 
\end{enumerate} 
\noindent
satisfy $(1) \Rightarrow (2) \Rightarrow (3) \Rightarrow (4)$.
\end{LEM}

\noindent
Notice that by Theorem \ref{th-intro-main1} the statements in Lemma \ref{lem-intro-subductive-weight-val} are all equivalent if $A$ is positively graded and $\B \subset A$ consists of homogeneous elements. 

\begin{REM}
In general, by Theorem \ref{th-Khovanskii-basis-equiv-conditions}, $\B$ is a Khovanskii basis of $\v$ if and only if certain ideals $I_M$ and $I_\v$ coincide. By Lemma \ref{lem-gr_v-in_w}, this implies that $\gr_\v(A) \cong \gr_{\v_M}(A)$, which is a weakening of $(3)$ above.  Furthermore, by Theorem \ref{th-Khovanskii-basis-equiv-conditions} this is equivalent to Algorithm \ref{algo-subduction} terminating on a specific collection of elements of $A$, which is a weakening of (1).  
\end{REM}

The next theorem is about constructing valuations from prime cones (see Section \ref{sec-val-from-prime-cone}).
\begin{THM} \label{th-intro-main2}
Let $C \subset \T(I)$ be a prime cone. Let ${\bf u} = \{\u_1, \ldots, \u_r\} \subset C$ be a collection of rational vectors that span a real vector space of maximal dimension $\dim(C)$. Let $\w \in \Q^{r \times n}$ be the matrix whose row vectors are $\u_1, \ldots, \u_r$. Let $\v_{\w}: A \setminus \{0\} \to \Q^r$ be the weight quasivaluation associated to $M$. Then $\v_M$ is a valuation and the following hold:
\begin{itemize}
\item[(1)] $\gr_{\v_{\w}}(A) \cong \k[x_1, \ldots, x_n]/\In_C(I)$. 
\item[(2)] $M$ coincides with the matrix:
$$
M_{\bf u} = \left[ \begin{array}{ccc}
\v_{\u_1}(b_1)  & \cdots & \v_{\u_1}(b_n)\\ 
\vdots & \vdots & \vdots \\
\v_{\u_r}(b_1) & \cdots & \v_{\u_r}(b_n) \end{array} \right],
$$ where, for every $i$, $\v_{\u_i}$ is the weight quasivaluation on $A$ associated to the weight vector $\u_i$.
\item[(3)] The value semigroup $S(A, \v_{\w})$ is generated by the columns of the matrix $M$.
{\item[(4)] If $C$ lies in the Gr\"obner region $\GR(I)$ the valuation $\v_M$ has an adapted basis which can be taken to be a standard monomial basis for a maximal cone in the Gr\"obner fan of $I$ containing $C$.}
\end{itemize}
\end{THM}
Here the Gr\"obner region $\GR(I)$ is taken to be the set of all $u \in \Q^n$ for which there is a term order $>$ with $\In_{>}(\In_u(I)) = \In_{>}(I)$ (see \ref{def-GR-r}). The Gr\"obner region always contains the negative orthant $\Q_{\leq 0}^n$. We remark that if $C$ is a maximal prime cone (i.e., it has dimension $d=\dim(A)$) then $\In_C(I)$ is a prime binomial ideal. 
We also note that by Theorem \ref{th-intro-main2}(1), the associated graded of the valuation $\v_{\w}$ only depends on the cone $C$. That is, given $C$, it is possible to find many valuations $\v_M$ on $A$ with the same associated graded algebra.  

As a corollary of Theorem \ref{th-intro-main2} we obtain the following. 
\begin{CORO}
When $A$ is positively graded and $\B$ consists of homogeneous elements of degree $1$, we can choose ${\bf u}$ so that the first row of $M$ is $(-1, \ldots, -1)$. One observes that in this case, after dropping a minus sign, the valuation $\v_{\w}$ is as in \eqref{equ-intro-v-degree} and the Newton-Okounkov body $\Delta(A, \v_{\w})$ is the convex hull of the $\v_{\w}(b_i)$. In particular, when the cone $C$ is maximal, the degree of $Y=\Proj(A)$ is equal to $\dim(Y)!$ times the volume of this convex hull.
\end{CORO}

Conversely, any weight valuation comes from a prime cone in the tropical variety in the following sense (see Section \ref{sec-cone-from-subductive-val}). 

\begin{THM}\label{th-intro-main3}
Let $\v = \v_M: A\setminus \{0\} \to \Q^r$ be a weight valuation with weighting matrix $M \in \Q^{r \times n}$ corresponding to a presentation $A \cong \k[x_1, \ldots, x_n] / I$ {(recall that we assume $\v_M(f) \neq \infty$ for all $0 \neq f \in A$)}. Then there is a prime cone $C_{\v} \subset \T(I)$ such that: $$\k[x_1, \ldots, x_n]/\In_{C_{\v}}(I) \cong \gr_{\v}(A).$$ 
In particular, in light of Lemma \ref{lem-intro-subductive-weight-val}, there exists a prime cone $C_\v$ for any subductive valuation $\v$ on $A$.  Furthermore, if $A$ is positively graded and $\B$ is a homogeneous generating set, there is a prime cone $C_\v$ associated to any valuation on $A$ with Khovanskii basis $\B$. 
\end{THM}

We remark that Theorem \ref{th-intro-main3} implies that only a finite number of associated graded algebras are possible among the weight valuations with fixed Khovanskii basis $\B$, and these are indexed by certain faces of $\T(I)$.  Similarly, if we restrict attention to positively graded algebras, only a finite number of associated graded algebras are possible among all valuations with Khovanskii basis $\B$.

In Section \ref{sec-compactifications}, we study a compactification $\bar{X}_{\bf u}$ of $X = \Spec(A)$ by boundary components associated to a linearly independent set ${\bf u}=\{u_1, \ldots, u_r\} \subset C$ where $C$ is a prime cone of dimension $r$ in the tropical variety $\T(I)$. We show that under mild conditions, the divisor $D_{\bf u} = \bar{X}_{\bf u} \setminus X$ is of combinatorial normal crossings type (\cite{Sturmfels-Tevelev}).

\begin{THM} \label{th-intro-compactification}
Let $X = \Spec(A)$ and $C$ be a prime cone whose relative interior intersects the negative part $\T(I)^- = \T(I) \cap \Q_{\leq 0}^n$ of the tropical variety of an ideal $I$. Let ${\bf u} = \{u_1, \ldots, u_t\} \subset C$ be a choice of $r$ linearly independent vectors. Let $\w \in \Q^{r \times n}$ be the corresponding weighting matrix, i.e. the rows of $M$ are $u_1, \ldots, u_r$, and let $\v_{\w}$be its associated valuation. Then there is a compactification $X \subset \bar{X}_{\bf u}$ of combinatorial normal crossings type whose boundary is a union of $r$ reduced, irreducible divisors $D_1, \ldots, D_r$. The valuation $\v_{\w}$ can be recovered from this boundary divisor in the sense that
for any member $b \in \mathbb{B} \subset A$ of an adapted basis $\mathbb{B}$, we have $\v_{\w}(b) = (\ord_{D_1}(b), \ldots,
\ord_{D_r}(b))$. 
\end{THM}

\begin{REM}
{A choice of ordering of the elements of ${\bf u}$ gives an ordering of the irreducible components $D_1, \ldots, D_r$.} This can be used to define a flag of subvarieties in 
$\bar{X}_{\bf u}$: 
$$ D_1 \cap \cdots \cap D_r \subset \cdots \subset D_1 \cap D_2 \subset D_1.$$
The valuation $\v_{\w}$ coincides with the valuation associated to the above flag of subvarieties (see \cite[Example 2.13]{KK} or \cite{LM} for the notion of the valuation associated to a flag of subvarieties).    
\end{REM}

\begin{REM}
The above construction includes some well-known compactifications, and in particular the wonderful compactification of an adjoint group (see Example \ref{ex-wonderful-comp}). More precisely, let $G$ be an adjoint group (over an algebraically closed characteristic $0$ field $\k$) with weight lattice $\Lambda$ and semigroup of dominant weights $\Lambda^+$. One defines a natural valuation on the coordinate ring $\k[G]$ as follows. Consider the isotypic decomposition $\k[G] = \bigoplus_{\lambda \in \Lambda^+}(V_\lambda \otimes V_\lambda^*)$ for the left-right $(G \times G)$-action.
Fix an ordering of fundamental weights. This defines a lexicographic ordering on the weight lattice $\Lambda$ of $G$. For 
$f \in \k[G]$ let us write $f = \sum_\lambda f_\lambda$ as the sum of its isotypic components and define $\v(f) = \min\{ \lambda \mid f_\lambda \neq 0\}$. One verifies that this gives a $(G \times G)$-invariant valuation $\v: \k[G] \setminus \{0\} \to \Lambda$. Also one can see that this valuation is of the form $\v_{\w}$ (as in Theorem \ref{th-intro-main2}) and comes with an associated compactification (as in Theorem \ref{th-intro-compactification}). Moreover, from Brion's description of the total coordinate ring of a wonderful compactification (\cite{Brion-total}) it follows that the compactification associated to $\v$ is the wonderful compactification of $G$. Theorem \ref{th-intro-compactification} then implies that the valuation $\v$ corresponds to a flag of $(G \times G)$-orbit closures in the wonderful compactification. {One can extend this example to other spherical homogeneous spaces.}
\end{REM}

\begin{REM}
In \cite{GHKK}, Gross, Hacking, Keel, and Kontsevich construct general toric degenerations in the context of cluster algebras and define and study related compactification constructions. 
This is also present in the related work of Rietsch and Williams on the Grassmannian variety $\Gr_k(\C^n)$ (\cite{RW}). The paper \cite{BFFHL} considers certain Khovanskii bases for the Pl\"ucker algebras of $\Gr_2(\C^n)$ and $\Gr_3(\C^3)$ in connection with \cite{RW}. We suspect that these constructions agree with variants of the 
ones we define here when the cone $C \subset \T(I)$ is chosen with regard to a Khovanskii basis of cluster monomials.
{More specifically, in light of the results in \cite{BFFHL}, we think that the valuation on Pl\"ucker algebra of $\Gr_2(\C^n)$ associated to a plabic graph (as constructed in \cite{RW}) coincides with the valuation constructed from a prime cone (in the tropical Grassmannian) as in Theorem \ref{th-intro-main2}.}

We would like to mention the recent paper \cite{KU} as well. In this paper the authors also make a connection, but in a quite different direction than ours, between the theory of Newton-Okounkov bodies and tropical geometry. 
\end{REM}

Finally, we say a few words about tropical sections and existence of finite Khovanskii bases. 
Recall that the tropical variety $\T(I)$ can be realized as the image of the Berkovich analytification $\Spec(A)^{\an}$ under a tropicalization map $\phi_{\B}$ (\cite{P}, \cite{MSt}, Section \ref{tropical}). It is of interest in tropical geometry to know when the tropicalization map $\phi_{\mathcal{B}}: \Spec(A)^{\an} \to \T(I)$ has a section $s: U \to X^{\an}$, for $U \subset \T(I)$. In \cite{GRW}, Gubler, Rabinoff, and Werner build off of work of Baker, Payne, and Rabinoff \cite{BPR} on curves, to show that such a section always exists over the locus of points with tropical multiplicity $1$. Our Theorem \ref{th-intro-main3} implies that the cone $C_{\v}$ corresponding to a subductive valuation $\v$ has such a section. In particular, following \cite[Section 10]{GRW}, a point $\u \in C_{\v} \subset \T(I)$ lies under a strictly affinoid domain in the analytification $\Spec(A)^{\an}$ with a unique Shilov boundary point (the weight valuation $\v_\u$).  In this sense, $C_\v$ can be regarded as a cone in $\Spec(A)^{\an}$.  With this in mind, we will explore the relationship between convex sets in the Berkovich analytification $\Spec(A)^{\an}$ and higher rank valuations in future work. 

Theorem \ref{th-intro-main2} states that a prime cone $C \subset \T(I)$ can be used to produce a discrete valuation with a prescribed Khovanskii basis. This leads us to the following problem.

\begin{PROB} \label{simpleconeexists}
Given a projective variety $Y$, find an embedding of this variety into a projective toric variety so that the resulting tropicalization contains a prime cone of maximal dimension. 
\end{PROB}


A positive resolution of Problem \ref{simpleconeexists} implies the existence of a toric degeneration.  Recent work of Ilten and Wr\"obel \cite{Ilten-Wrobel} shows that it is not always possible to find a full dimensional prime cone.  In forthcoming work \cite{KMM} the authors and Takuya Murata show that it is always possible to find a prime cone of dimension one less than the Krull dimension when the domain $A$ is positively graded. In particular this implies that any projective variety caries a flat degeneration to a complexity-one $T$-variety.

\begin{REM}[Polyhedral Newton-Okounkov bodies]
Let us consider the graded case $A = \bigoplus_{i \geq 0} A_i$ and let us take a valuation $\v$ as in \eqref{equ-intro-v-degree} which encodes the degree function on $A$. It is immediate from the definition that if the value semigroup $S(A, \v)$ is finitely generated then the corresponding Newton-Okounkov body $\Delta(A, \v)$ is a rational polytope. But it is easy to see that the other implication does not always hold, i.e. if $\Delta(A, \v)$ is a rational polytope it does not imply that $S(A, \v)$ is finitely generated.  
In fact,  by the work \cite{AKL} (see also \cite{Seppanen}) one knows that for homogeneous coordinate rings of projective varieties (and more generally rings of sections of big line bundles) one can find valuations such that the corresponding Newton-Okounkov bodies are rational simplices.
\end{REM}

Using results of Gubler, Rabinoff, and Werner \cite{GRW}, a resolution of Problem \ref{simpleconeexists} appears possible in the case that $U \subset \Gm^n$ is a very affine variety and $\k$ a field of characteristic $0$. First, one constructs a compactification $\bar{X} \supset U$, and resolves it to a smooth normal crossings compactification using Hironaka's strong resolution of singularities in characteristic $0$ (\cite{Hironaka}, \cite[Theorem 3.3 ]{Kollar}). 
By \cite[Theorems 8.4 and 9.5]{GRW}  (see also  \cite{Cueto}), this compactification produces a tropical skeleton in $\bar{X}^{\an}$, along with an  {embedding of an open subset of $U$} whose tropicalization ``sees'' this skeleton as a set of points with tropical multiplicity $1$.  This tropicalization then contains a prime cone $C$. 
With this in mind, it would be interesting to have a solution of the following problem.

\begin{PROB}
Let $A$ be a finitely generated $\k$-algebra and domain with Krull dimension $d$. Find an effective algorithm for constructing a valuation $\v: A \setminus \{0\} \to \Q^d$ of maximal rank $d$ and with a finite Khovanskii basis. 
\end{PROB}

In Section \ref{sec-examples} we consider some examples of the main results of the paper. These include the Gel'fand-Zetlin bases for homogeneous coordinate rings of the Pl\"ucker embeddings of the Grassmannians and flag varieties. In this regard, we would also like to mention the related work \cite{StXu} on Cox rings of del Pezzo surfaces. 

\bigskip          
\noindent{\bf An Example.}
Let $E \subset \mathbb{P}^2$ be the {elliptic curve} cut out by the homogeneous equation 

\begin{equation}
y^2z - x^3 +7xz^2 -2z^3 =0.\\
\end{equation}
Let us illustrate how to construct a subductive valuation (in particular with a finite Khovanskii basis) for the homogeneous coordinate ring of $E$ using prime cones in its tropical variety as in Section \ref{sec-val-from-prime-cone}. The tropical variety $\T$ of $y^2z - x^3 +7xz^2 -2z^3$ is the union of the three half-planes $\Q(1, 1,1) + \Q_{\geq 0} (1,0 ,0 )$, $\Q(1, 1,1) + \Q_{\geq 0} (0, 1 ,0 )$, and $\Q(1, 1,1) + \Q_{\geq 0} (-2,-3 ,0 )$ with initial forms $y^2z - 2z^3$, $-x^3 +7xz^2 - 2z^3$, and $y^2z - x^3$, respectively (see Figure \ref{ATROP}). 

\begin{figure}[htbp]
\centering
\includegraphics[scale = 0.12]{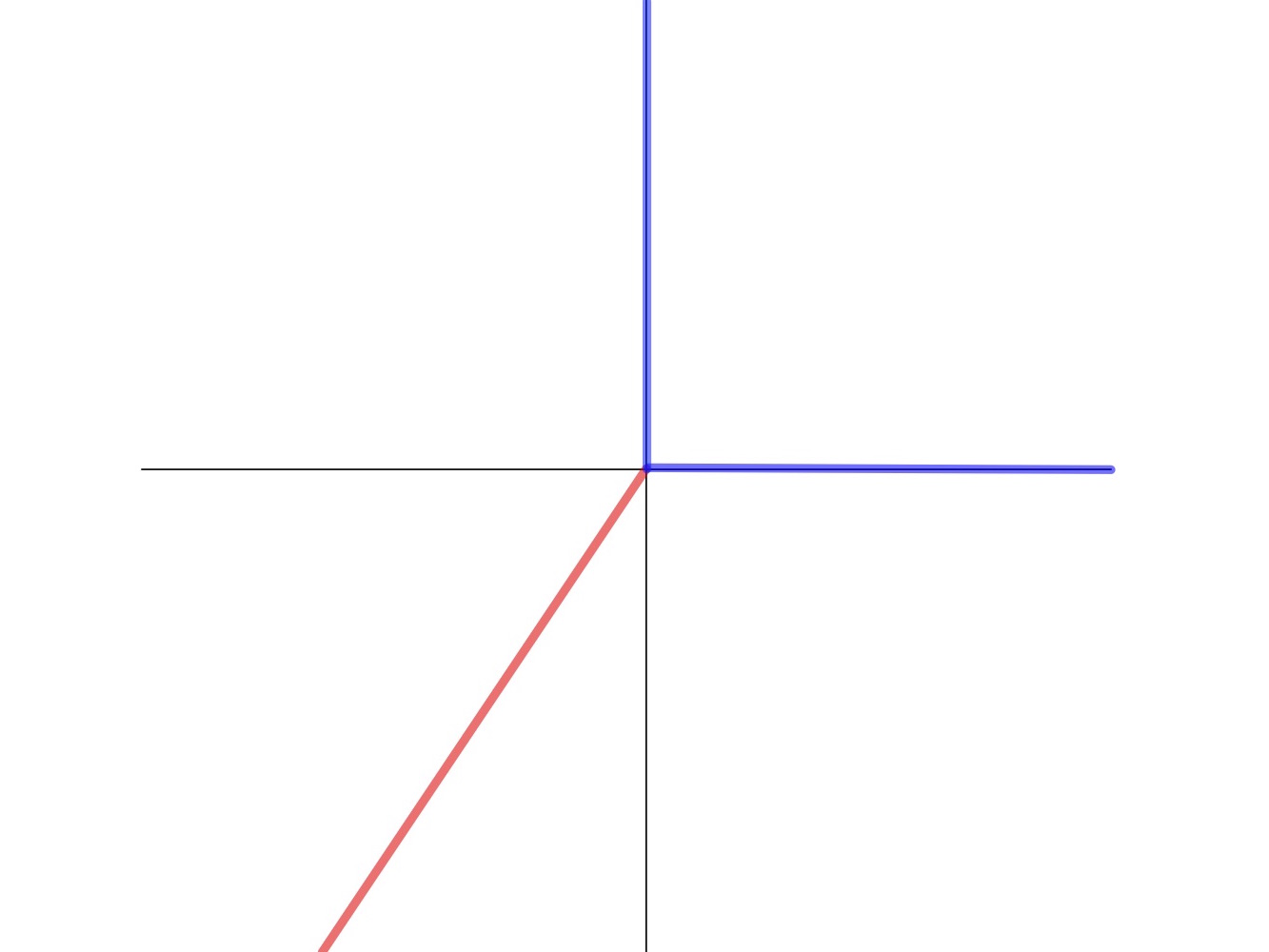}
\caption{The tropical variety $\T/\Q(1, 1, 1)$. The image of the cone $C$ is in the negative orthant.}
\label{ATROP}
\end{figure}

The half-plane $C = \Q(1, 1,1) + \Q_{\geq 0} (-2,-3 ,0 )$ is the only prime cone, and by Theorem \ref{th-intro-main1} it can be used to create  a subductive valuation $\v: \k[E]\setminus \{0\} \to \Z^2$.  Using Section 
\ref{sec-val-from-prime-cone} we can construct this valuation by sending $x, y, z$ to the first, second and third columns of the following weighting matrix $M$ respectively:

\begin{equation}
M = \left[ \begin{array}{ccc}
-1  & -1 & -1 \\
-2 & -3 & 0 \end{array}\right]. \\
\end{equation}

We have obtained $M$ by taking its rows to be the vectors $(-1,-1,-1)$ and $(-2,-3,0) \in C$.  This assignment is then extended linearly to all monomials in $x, y, z$, and the resulting set is lexicographically ordered. As a consequence the value semigroup $S(\k[E], \v)$ is the $\Z_{\geq 0}$-span of the columns of $M$. The Newton-Okounkov cone $P(\k[E], \v)$ is the convex hull of 
$S(\k[E], \v)$ (see Figure \ref{ANOK1}).  
\begin{figure}[htbp]
\centering
\includegraphics[scale = 0.13]{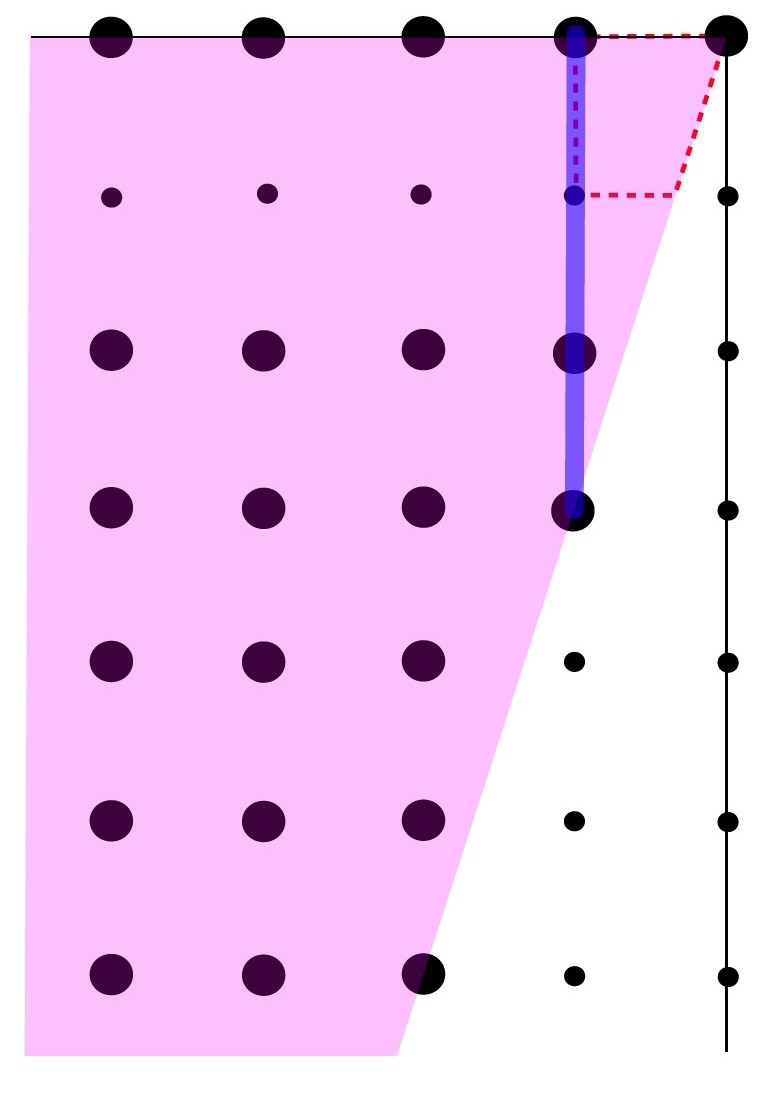}
\caption{The Newton-Okounkov cone of $\k[E]$ with highlighted Newton-Okounkov body. The larger dots indicate members of $S(\k[E], \v)$.}
\label{ANOK1}
\end{figure}
The Newton-Okounkov body $\Delta(\k[E], \v)$ is the convex hull of the columns of $M$, this is an interval of length $3 = \deg(y^2z - x^3 +7xz^2 -2z^3)$ in $\{-1\} \times \R$.

From Section \ref{sec-compactifications} we obtain a compactification $\bar{X}$ of the affine cone $\hat{E} = \Spec(\k[E])$ associated to the choice of matrix $M$.  The projective coordinate ring $\k[\bar{X}]$ given by this construction is presented by $7$ parameters: $T$ and $z$ have homogeneous degree $1$, $x$ and $X$ have homogeneous degree $2$, and finally $y, Y,$ and $Z$ have homogeneous degree $3$.  The ideal which vanishes on these parameters is generated by the following forms:

\begin{equation}
xY - Xy, \ \ \ xZ - XY, \ \ \ TX - xz, \ \ \ TY - yz,  \ \ \ TZ - Yz,\\
\end{equation}

$$X^3 - Z^2 +7Xz^4 -2z^6, \ \ \ xX^2 - yZ +7Xz^3T - 2z^5T$$

$$x^2X - yZ +7Xz^2T^2 -2z^4T^2, \ \ \ x^2X - Z^2 +7Xz^2T^2 - 2z^4T^2, \ \ \ x^3z - yY +7XzT^3 -2z^3T^3.$$
The polytope bordered by the dotted line in Figure \ref{ANOK1} is the Newton-Okounkov polytope of the compactification that appears in Section \ref{sec-compactifications} (see \eqref{equ-Delta-u}). The above relations were obtained by lifting a \emph{Markov basis} of a certain toric ideal. 

The compactifying divisor $D = \bar{X} \setminus \hat{E}$ is the locus of the equation $T = 0$. It has two components $D_1, D_2$ cut out by the ideals $I_1 = \langle T, z \rangle$ and $I_2 = \langle T, x, y, Y \rangle$ respectively.   

\bigskip
\noindent{\bf Acknowledgement.} 
We would like to thank Diane Maclagan, Bernd Sturmfels, Dave Anderson, Alex K\"uronya and Michel Brion for useful comments on a preliminary version of this paper. In fact, this paper was inspired by a question that Bernd asked the authors about a connection between toric degenerations arising in tropical geometry and the theory of Newton-Okounkov bodies. We also thank Angie Cueto, Joseph Rabinoff, and Igor Makhlin for explaining their work to us. In particular, we thank Igor Makhlin for pointing out an error in Example \ref{GZpatternexample} in an earlier draft of this paper.  The authors are also grateful to the Fields Institute for Research in Mathematical Sciences for hosting the thematic program Combinatorial Algebraic Geometry. Some of this work was done during this program. Finally, thanks to Sara Lamboglia and Kalina Mincheva and several other participants of the Apprenticeship Program at the Fields Institute for helpful discussions.

\bigskip
\noindent{\bf Notation.}
Throughout the paper we use the following notation.
\begin{itemize}
\item $\k$, a field which we take to be our base field throughout the paper.
\item $\k[\x]$, the polynomial ring over $\k$ associated to a finite set of indeterminates $\x = (x_1, \ldots, x_n)$.
\item $A$, a finitely generated $\k$-algebra and domain with Krull dimension $d = \dim(A)$. We sometimes assume that $A$ is positively graded i.e. graded by $\Z_{\geq 0}$.
\item $\v: A \setminus \{0\} \to \Q^r$, a discrete valuation on $A$ (see Definition \ref{subsec-valuation}). We denote the corresponding associated graded by $\gr_\v(A)$. The case of main interest is when $\v$ has one-dimensional leaves, which in turn implies that it has full rank $d = \dim(A)$.
\item $S(A, \v)$, the value semigroup of $(A, \v)$ and $P(A, \v)$ the Newton-Okounkov cone of $(A, \v)$, i.e. the closure of the convex hull of $S(A, \v)$. Also when $A$ is positively graded the corresponding Newton-Okounkov body is denoted by $\Delta(A, \v)$ (see Sections \ref{subsec-valuation} and \ref{subsec-NO-bodies}).
\item $\B = \{b_1, \ldots, b_n\}$ a set of $\k$-algebra generators for $A$. 
\item $I \subset \k[\x]$, an ideal usually taken to be the kernel of a surjective homomorphism $\pi: \k[\x] \to A$ given by $x_i \mapsto b_i$. We refer to $\pi$ or $\k[\x] / I \cong A$ as a presentation of $A$.
\item $\w \in \Q^{r \times n}$, a weighting matrix for the parameters in $\x$, inducing an initial term valuation $\tilde{\v}_\w: \k[\x] \setminus \{0\} \to \Q^r$ (see Section \ref{subsec-weight-quasival}).
\item $\v_\w: A \setminus \{0\} \to \Q^r$, the quasivaluation on $A$ obtained by pushforward of $\tilde{\v}_\w$ via the homomorphism $\pi$ (see Definition \ref{def-quasival} 
and Section \ref{subsec-weight-quasival}). 
\item $\textup{GR}(I)$, the Gr\"obner region of an ideal $I$. It is taken to be the set of all $u \in \Q^n$ for which there is a term order $>$ with $\In_{>}(\In_u(I)) = \In_{>}(I)$.
\item $\Sigma(I)$, the Gr\"obner fan of an ideal $I$.
\item $\T(I)$, the tropical variety of an ideal $I$ (see Definition \ref{def-trop-var}).
\end{itemize}

\section{Valuations on algebras and Khovanskii bases}   \label{sec-valandKhovanskii}
In this section we introduce some of the basic terminology and results concerning valuations and we develop a general theory of Khovanskii bases. 

\subsection{Preliminaries on valuations}  \label{subsec-valuation}
Throughout the paper, a (linearly) ordered group is an abelian group $(\Gamma, +)$ equipped with a total ordering $\succ$ which respects the group operation. 
Primarily we work with a discrete subgroup of a rational vector space $\Q^r$ with some linear ordering. By the Hahn embedding theorem (\cite{Gravett}), there is always an embedding of  linearly ordered groups $\eta: \Q^r \to \R^r$, where $\R^r$ is given the standard lexicographic ordering, consequently we may treat any linear ordering by considering the lexicographic case.\footnote{Also it might be worthwhile mentioning the related classical result that any monomial order on $\Z^r$ is obtained from the lexicographic order and a real-valued weighting matrix.}  

Let $(\Gamma, \succ)$ be a linearly ordered group.
\begin{definition}[Valuation] \label{def-valuation}
We recall that a function $\v: A \setminus \{0\} \to \Gamma$ is a {\it valuation} over $\k$ if it satisfies the following axioms:
\begin{itemize}
\item[(1)] For all $0 \neq f, g \in A$ with $0 \neq f+g$ we have $\v(f+g) \succeq \MIN\{ \v(f), \v(g)\}$. Here $\MIN$ is computed using $\succ$.
\item[(2)] For all $0 \neq f, g \in A$ we have $\v(fg) = \v(f) + \v(g)$.
\item[(3)] For all $0 \neq f \in A$ and $0 \neq c \in \k$ we have $\v(cf) = \v(f)$.  
\end{itemize}

Each valuation $\v$ on $A$ naturally gives a $\Gamma$-filtration $\mathcal{F}_\v = (F_{\v \succeq a})_{a \in \Gamma}$ on $A$. Namely, for $a \in \Gamma$ we define:
$$F_{\v \succeq a} = \{ f \in A \setminus \{0\} \mid \v(f) \succeq a\} \cup \{0\}.$$
($F_{\v \succ a}$ is defined similarly.) Clearly $F_{\v \succeq a}$ and $F_{\v \succ a}$ are vector subspaces of $A$. The corresponding associated graded is $\gr_\v(A) = \bigoplus_{a} F_{\v \succeq a} / F_{\v \succ a}$.

If the following extra property is satisfied we say that $\v$ has {\it one-dimensional leaves}:
\begin{itemize}
\item[(4)] For every $a \in \Gamma$ the quotient vector space:
$$F_{\v \succeq a} / F_{\v \succ a},$$
is at most $1$-dimensional.
\end{itemize}
\end{definition}

Let $K$ denote the quotient field of $A$. Let $R_\v = \{ f \in K \setminus \{0\} \mid \v(f) \succeq 0\} \cup \{0\}$, $\mathfrak{m}_\v = \{ f \in K \setminus \{0\} \mid \v(f) \succ 0\} \cup \{0\}$ and $\k_\v = R_\v / \mathfrak{m}_\v$ denote the valuation ring of $\v$, its maximal ideal and its residue field respectively. Clearly $\k_\v$ contains $\k$. It is straightforward to verify that a valuation $\v$ has one-dimensional leaves if and only if the residue field extension is trivial, that is, $\k_\v = \k$.

Below we give some examples of valuations. These contain some cases of interest in computational algebra, algebraic geometry and representation theory, and partly motivated the present work.

\begin{example}  \label{ex-valuation}  
(1) Let $A$ be graded by an ordered group $\Gamma$, i.e. $A = \bigoplus_{g \in \Gamma} A_g$. Using the $\Gamma$-grading we can define a valuation $\v: A \setminus \{0\} \to \Gamma$ as follows.
Take $0 \neq f \in A$ and let $f = \sum_{g \in \Gamma} f_g$ be its decomposition into homogeneous components. We then define:
$$\v(f) = \MIN\{g \mid f_g \neq 0\}.$$
We call the valuation $\v$ constructed in this way a {\it grading function}.  It is easy to see that the associated graded algebra $\gr_{\v}(A)$ is canonically isomorphic to $A$. 

(2) Consider the algebra of polynomials $\k[\x]$ in $n$ indeterminates $\x=(x_1, \ldots, x_n)$. It is graded by the semigroup $\Z_{\geq 0}^n \subset \Z^n$. Fix a group ordering $\succ$ on $\Z^n$. As a particular case of the part (1) above, $\succ$ gives rise to a valuation $\v: \k[\x] \setminus \{0\} \to \Z_{\geq 0}^n$. We call it the {\it lowest term or minimum term valuation}. One verifies that $\v$ is a valuation with one-dimensional leaves.

(3) More generally, let $X$ be a $d$-dimensional variety defined over $\k$.  Take a smooth point $p \in X$ and let $u_1, \ldots, u_d$ be a system of local parameters at $p$. Every rational function $f$ regular at $p$ can be expressed as a power series in the $u_i$. Fixing a group ordering on $\Z^d$, one can define $\v(f)$ as the minimum exponent appearing in the power series of $f$. This extends to define a valuation (with one-dimensional leaves) on the field of rational functions $K$ of $X$. 

(4) Yet more generally, instead of a system of parameters at a smooth point, one can associate a valuation to a flag of subvarieties in $X$ (see \cite[Example 2.13]{KK} and \cite{LM}).     

(5) Let $G$ be a connected reductive algebraic group defined over an algebraically closed characteristic $0$ field $\k$. Let $X$ be an affine variety equipped with a $G$-action. Let $\Lambda$ be the weight lattice of $G$ and $\Lambda^+$ the subsemigroup of dominant weights. One can decompose the coordinate ring $A = \k[X]$ as a direct sum $\bigoplus_{\lambda \in \Lambda^+} A_\lambda$, where $A_\lambda$ is the $\lambda$-isotypic component of $A$, i.e. the sum of irreducible representations in $A$ with highest weight $\lambda$. Fix a group ordering $\succ$ on $\Lambda$. One usually would like to assume that this ordering refines the so-called dominant partial order. Given $f \in A$ let us write $f = \sum_\lambda f_\lambda$ with $f_\lambda \in A_\lambda$. One can then define $\v(f) = \MIN\{ \lambda \mid f_\lambda \neq 0\}$ where the minimum is with respect to $\succ$. This defines a valuation $\v: A \setminus \{0\} \to \Lambda^+$. This valuation in general does not have one-dimensional leaves property.
\end{example}


For simplicity, in this section, 
we consider $\Gamma$ to be the additive group $\Z^r$, for some $0 < r \leq d$, equipped with a linear ordering $\succ$ (e.g. a lexicographic order).

We denote by $S = S(A, \v)$ the {\it value semigroup} of $(A, \v)$, namely:
\begin{equation} \label{equ-value-semigp}
S = \{ \v(f) \mid 0 \neq f \in A \}.
\end{equation}
Clearly $S$ is an (additive) subsemigroup of $\Z^r$.
The {\it (rational) rank of the valuation $\v$} is the rank of the sublattice of 
$\Z^r$ generated by $S(A, \v)$.

The following theorem shows that when $\k$ is algebraically closed, and the valuation $\v$ has full rank $d = \dim(A)$, then it automatically has 
one-dimensional leaves. It is an immediate corollary of Abhyankar's inequality (see \cite[Theorem 6.6.7]{Huneke-Swanson} for statement of Abhyankar's inequality).

\begin{theorem} \label{th-full-rank-1-dim-leaves}
Let $\k$ be algebraically closed and assume that $\v$ has full rank $d = \dim(A)$. Then $\v$ has one-dimensional leaves.
\end{theorem}

The next proposition states that if $\v$ is assumed to have one-dimensional leaves property (Definition \ref{def-valuation}(4)) then the associated graded algebra $\gr_\v(A)$ can be realized as the semigroup algebra of the value semigroup $S=S(A, \v)$. We omit the proof here (see \cite[Remark 4.13]{Bruns-Gubeladze}).
\begin{proposition} \label{prop-grA-semigp-algebra}
Let $A$ be a domain. If $\v$ has one-dimensional leaves property then $\gr_\v(A)$ is isomorphic to the semigroup algebra $\k[S]$ (note that we do not require $S$ to be finitely generated). More generally, if $R$ is a $\Z^d$-graded algebra such that 
for every $a \in \Z^d$, the corresponding graded piece $R_a$ is at most $1$-dimensional, then $R$ is isomorphic to the semigroup algebra 
$\k[S]$ where $S$ is the subsemigroup of $\Z^d$ defined by $S = \{ a \in \Z^d \mid R_a \neq \{0\} \}$. 
\end{proposition}

\subsection{Khovanskii bases and subduction algorithm}  \label{subsec-Khovanskii-basis}
The following definition is one of the key definitions in the paper. It generalizes the notion of a SAGBI basis (also called a canonical basis in \cite{St}) for 
a subalgebra of a polynomial ring (\cite{SAGBI}). Recall that $\v: A \setminus \{0\} \to \Q^r$ is a discrete valuation where $0 < r \leq d = \dim(A)$.
\begin{definition}[Khovanskii basis] \label{def-Khovanskii-basis}
We say that $\mathcal{B} \subset A$ 
is a {\it Khovanskii basis} for $(A, \v)$ if the image of $\B$ in $\gr_\v(A)$ is a set of algebra generators for $\gr_\v(A)$. Note that we do not require $\mathcal{B}$ to be finite although the case of main interest is when it is finite.
\end{definition}
The name Khovanskii basis was suggested by B. Sturmfels in honor of A. G. Khovanskii's influential contributions to combinatorial commutative algebra. 

\begin{remark}  \label{rem-SAGBI}
Let $A$ be a subalgebra of a polynomial algebra $\k[\x]$. A Khovanskii basis for a lowest term valuation, as in Example \ref{ex-valuation}(2), is usually called a {\it SAGBI basis}, which stands for {\it Subalgebra Analogue of Gr\"obner Basis for Ideals} (see \cite{SAGBI}, \cite[Chapter 11]{St}). So the theory of Khovanskii bases far generalizes that of SAGBI bases.
\end{remark}

Below are two examples of algebras with valuations which have finite Khovanskii bases.
\begin{example}   \label{ex-Khov-basis}
(1) Take the standard lexicographic order $\succ$ on $\Z^n$, that is, $e_1 \succ \cdots \succ e_n$ where $\{e_1, \ldots, e_n\}$ is the standard basis. Let $\v$ be the lowest term valuation on the polynomial algebra $\k[\x]$ as defined in Example \ref{ex-valuation}(2). 
Let $A = \k[\x]^{S_n}$ be the subalgebra of symmetric polynomials. It is well-known that this algebra is freely generated by the elementary symmetric polynomials. One verifies that the value semigroup $S(A, \v)$ is $\{(a_1, \ldots, a_n) \in \Z_{\geq 0}^n \mid a_1 \leq \cdots \leq a_n\}$ which is a finitely generated semigroup. In fact, the elementary symmetric polynomials form a finite Khovanskii basis for $(A, \v)$.

(2) Let $G$ be a connected reductive algebraic group over an algebraically closed characteristic $0$ field $\k$ and $X$ an affine variety with an algebraic action of $G$. As in Example \ref{ex-valuation}(5), let $\v$ be the valuation on the coordinate ring $A = \k[X]$ and with values in the weight lattice $\Lambda$ (with respect to a group ordering $\succ$ on $\Lambda$). One shows that if $\succ$ refines the so-called dominant partial order on $\Lambda$ then the associated graded $\gr_\v(A)$ is the so-called horospherical degeneration of $A$. This is known to be a finitely generated algebra and thus $(A, \v)$ has a finite Khovanskii basis. As mentioned before, in general the valuation $\v$ does not have full rank. In \cite[Section 8]{Kaveh-Crystal}, it is shown that when $X$ is a spherical $G$-variety then the valuation $\v$ can be naturally extended to a full rank valuation $\tilde{\v}$ on $A$ such that the semigroup $S(A, \tilde{\v})$ is finitely generated. In other words, $(A, \tilde{\v})$ also has a finite Khovanskii basis. This recovers the toric degeneration results in \cite{Caldero, AB, Kaveh-Spherical}.
\end{example}

\begin{remark}
The idea behind the definition of a Khovanskii basis is to reduce computations in the algebra $A$ to computations in $\gr_\v(A)$. The algebra $\gr_v(A)$ can be regarded as a degeneration of $A$ and in principle has a simpler structure than that of $A$, for example, it is graded by the semigroup $S = S(A, \v) \subset \Q^r$. The case of main interest is when $\v$ has one-dimensional leaves in which case $\gr_\v(A) \cong \k[S]$ is a semigroup algebra (Proposition \ref{prop-grA-semigp-algebra}). Doing computation in the algebra $\k[S]$ is more or less equivalent to doing computation in the semigroup $S$ which we regard as a combinatorial object.
\end{remark}



Here are two examples where the value semigroup $S(A, \v)$ and hence the associated graded $\gr_\v(A)$ are not finitely generated.
\begin{example} \label{ex-Gobel-elliptic}
(1)(G\"obel) Consider the polynomial algebra $\k[x_1, x_2, x_3]$. As in Example \ref{ex-valuation}(2), let $\v$ be the 
lowest term valuation with respect to the lexicographic order $e_3 \succ e_2 \succ e_1$. Let $A = \k[x_1, x_2, x_3]^{A_3}$ be the subalgebra of invariants of the alternating group $A_3$. One shows that the value semigroup $S(A, \v) \subset \Z_{\geq 0}^3$ is not finitely generated and hence $(A, \v)$ does not have a finite Khovanskii basis (see \cite{Gobel} and also \cite[Example 11.2]{St}).

(2) Let $A$ be the homogeneous coordinate ring of an elliptic curve $X$ sitting in $\mathbb{P}^2$ as the zero set of a cubic polynomial (in the Weierstrass form). Let $\v' = \ord_p: A \setminus \{0\} \to \Z$ be the order of vanishing valuation at a general point $p \in X$ and let $\v: A \setminus \{0\} \to \Z_{\geq 0} \times \Z$ be the valuation constructed out of $\v'$ and degree (as in \eqref{equ-intro-v-degree}). One verifies that $S(A, \v) = \{(i, a) \mid i \in \Z_{\geq 0},\, 0 \leq a < 3i \}$ which is not a finitely generated semigroup. On the other hand, if we take $p$ to be the point at infinity then this semigroup can be seen to be finitely generated (see \cite[Example 1.7]{LM} and \cite[Example 6]{Anderson}).
\end{example}

We will use the following notation. For $0 \neq h \in A$ we let $\bar{h}$ denote its image in $\gr_\v(A)$, i.e. the image of $h$ in the quotient $F_{\v \succeq a} / F_{\v \succ a}$ where $a = \v(h)$. Clearly, $\bar{h}$ is a homogeneous element with degree $a$. The next lemma shows that from a Khovanskii basis one can recover the value semigroup $S(A, \v)$ (for a valuation $\v$ with one-dimensional leaves property this also follows from Proposition \ref{prop-grA-semigp-algebra}). 

\begin{lemma}  \label{lem-Khov-basis-value-semigp}
Let $\B$ be a Khovanskii basis for $(A, \v)$. Then the set of values $\{ \v(b) \mid b \in \B\}$ generates $S(A, \v)$ as a semigroup.
\end{lemma}
\begin{proof}
Recall that $\gr_\v(A)$ is an $S(A, \v)$-graded algebra.
Let $0 \neq f \in A$ with $\v(f) = a$. Since $\B$ is a Khovanskii basis we can write $\bar{f}$ as a polynomial $\sum_{\alpha=(\alpha_1, \ldots, \alpha_n)} c_\alpha \bar{b}_1^{\alpha_1} \cdots \bar{b}_n^{\alpha_n}$, for some $b_1, \ldots, b_n \in \B$. Moreover, since $\bar{f}$ and the $\bar{b}_i$ are homogeneous, we can assume that for every $\alpha$, with $c_\alpha \neq 0$, the corresponding term $c_\alpha \bar{b}_1^{\alpha_1} \cdots \bar{b}_n^{\alpha_n}$ has degree $a$. That is, $a = \sum_i \alpha_i \v(b_i)$. This finishes the proof.
\end{proof}

Whenever we have a Khovanskii basis $\mathcal{B}$, we can represent the elements of the algebra $A$ as polynomials in the elements of $\mathcal{B}$ 
using a simple classical algorithm usually known as the {\it subduction algorithm}. 

\begin{algorithm}[Subduction algorithm] \label{algo-subduction}
{\bf Input:} A Khovanskii basis $\B \subset A$ and an element $0 \neq f \in A$. 
{\bf Output:} A polynomial expression for $f$ in terms of a finite number of elements of $\mathcal{B}$.
\begin{itemize}
\item[(1)] Since the image of $\B$ in $\gr_\v(A)$ generates this algebra, we can find $b_1, \ldots, b_n \in \B$ and a polynomial $p(x_1, \ldots, x_n)$ such that $\bar{f} = p(\bar{b}_1, \ldots, \bar{b}_n)$. Thus we either have $f = p(b_1, \ldots, b_n)$ or
$\v(f - p(b_1, \ldots, b_n)) > \v(f)$. 
\item[(2)] If $f = p(b_1, \ldots, b_n) $ we are done. Otherwise replace $f$ with $f - p(b_1, \ldots, b_n)$ and go to the step (1).
\end{itemize}
\end{algorithm}

\begin{example}  \label{ex-subduction-not-terminate}
In general, it is possible that the subduction algorithm does not terminate. For example, take $A = \k[x]$ to be the polynomial algebra in one variable $x$ and let $\v$ be the order of divisibility by $x$. As a Khovanskii basis take $\B = \{x+x^2\}$. Then the subduction algorithm never stops for $f = x$.
\end{example}

We have the following easy but useful proposition.
\begin{proposition} \label{prop-subduction}
Suppose the value semigroup $S=S(A, \v)$ is maximum well-ordered, i.e. every subset of $S$ has a maximum element with respect to the total order $\succ$. 
Then for any $0 \neq f \in A$  the subduction algorithm (Algorithm \ref{algo-subduction})
terminates after a finite number of steps.
\end{proposition}

A large class of examples where the maximum well-ordered assumption is satisfied are homogeneous coordinate rings of projective varieties. 
Below are some general situations where one can guarantee termination of subduction algorithm in finite time. 
\begin{example} \label{ex-graded-alg-max-well-ordered}
(1) Let $A$ be positively graded. 
Let $\v: A \setminus\{0\} \to \Z^r$ be a valuation on $A$ which refines the degree. That is, for any $0 \neq f_1, f_2 \in A$, $\deg(f_1) < \deg(f_2)$ implies that $\v(f_1) \succ \v(f_2)$ (note the switch). One shows that under these assumptions the value semigroup $S(A, \v)$ is maximum well-ordered. 

(2) Let $A = \bigoplus_{g \in \Gamma} A_g$ be an algebra graded by an abelian group $\Gamma$ and such that for every $g \in \Gamma$, $\dim_\k(A_g) < \infty$. Let $\v$ be a valuation on $A$ and $\B$ a Khovanskii basis for $(A, \v)$ consisting of $\Gamma$-homogeneous elements. Then the subduction algorithm terminates for any $0 \neq f \in A$.
\end{example}

For the rest of this subsection we assume that $S(A, \v)$ is maximum well-ordered and hence the subduction algorithm for $(A, \v, \B)$ always terminates.

It is a desirable situation to have a finite Khovanskii basis. Below we explain how to find a Khovanskii basis provided that we know such a basis exists
(Algorithm \ref{algo-find-Khovanskii-basis}). 
Before we present the algorithm, we need some preparation. The next lemma and theorem give a necessary and sufficient condition for a set of algebra generators to be a Khovanskii basis. These are extensions of similar statements from \cite[Chapter 11]{St} to the setup of Khovanskii bases.

Let $\B = \{b_1, \ldots, b_n\} \subset A$ be a subset that generates $A$ as an algebra. Let $a_i = \v(b_i)$, $i=1, \ldots, n$ and put $\A = \{a_1, \ldots, a_n\}$.
Let $\k[\x]$ denote the polynomial algebra in indeterminates $\x = (x_1, \ldots, x_n)$. Consider the surjective homomorphism $\k[\x] \to A$ given by $x_i \mapsto b_i$ and let $I$ be the kernel of this homomorphism. Also we consider the homomorphism $\k[\x] \to \gr_\v(A)$ given by $x_i \mapsto \bar{b}_i$, $i=1, \ldots, n$, where as before $\bar{b}_i$ denotes the image of $b_i$ in $\gr_\v(A)$. 
We denote the kernel of the homomorphism $\k[\x] \to \gr_\v(A)$ by $I_\v$.

\begin{remark} \label{rem-I_A-toric-ideal}
If we assume that the valuation $\v$ has one-dimensional leaves, then by Proposition \ref{prop-grA-semigp-algebra}, the image of the homomorphism $\k[\x] \to \gr_\v(A)$ is isomorphic to the semigroup algebra $\k[S']$ where $S'$ is the semigroup generated by the values $\v(b_i)$, $i=1, \ldots, n$. Thus, we see that the ideal $I_\v$ is a toric ideal and hence generated by binomials. 
When $\B$ is a Khovanskii basis, the semigroup $S'$ coincides with the whole value semigroup $S=S(A, \v)$ and $\k[\x] / I_\v \cong \k[S]$.
\end{remark}

Let $M$ be the $r \times n$ matrix whose columns are the vectors $\v(b_1), \ldots, \v(b_n)$.
Using $M$ we define a partial order on the group $\Q^n$ as follows.
Given $\alpha, \beta \in \Q^n$ we say $\alpha \succ_\w \beta$ if $M\alpha \succ M\beta$, where $\succ$ in 
the right-hand side is the total order on $\Q^r$ used in the definition of the valuation $\v$. We note that since in general $M$ is not a square matrix and hence not invertible, 
it can happen that $\alpha \neq \beta$ but $M\alpha = M\beta$. In this case, $\alpha$, $\beta$ are incomparable in the partial order $\succ_\w$. We can define the notion of initial form of 
a polynomial with respect to $\succ_\w$. Let $p(\x) = \sum_\alpha c_\alpha \x^\alpha \in \k[\x]$ be a polynomial. Let $m = m(p) = \MIN\{M \alpha \mid c_\alpha \neq 0 \}$ where the minimum is with 
respect to the total order $\succ$. We define the initial form $\In_\w(p) \in \k[\x]$ by
\begin{equation} \label{equ-in_w}
\In_\w(p)(\x) = \sum_{\beta} c_\beta \x^\beta,
\end{equation}
where the sum is over all the $\beta$ with $M\beta = m$. If $p = \In_M(p)$ we say that $p$ is {\it $M$-homogeneous}. We let $\In_\w(I)$ be the ideal of $\k[\x]$ generated by $\In_\w(p)$, $\forall p \in I$.  The initial form and the initial ideal are important constructions in Section \ref{sec-groebner}. One makes the following observation: 
\begin{lemma} \label{lem-in_w-I_B}
The ideal $\In_\w(I)$ is contained in the ideal $I_\v$.
\end{lemma}
\begin{proof}
Let $p(\x) = \sum_\alpha c_\alpha \x^\alpha \in I$, i.e. $p(b_1, \ldots, b_n) = 0$.  We 
note that for any monomial $c_\alpha \x^\alpha$, its valuation $\v(c_\alpha \x^\alpha)$ is given by
\begin{equation} \label{equ-val-monomial}
\v(c_\alpha b_1^{\alpha_1} \cdots b_n^{\alpha_n}) = M \alpha,
\end{equation}
where $\alpha = (\alpha_1, \ldots, \alpha_n)$. From \eqref{equ-val-monomial} and the non-Archimedean property of $\v$ (Definition \ref{def-valuation}(1))
we see that $\v(\In_\w(p)(b_1, \ldots, b_n)) \succ m = m(p)$. Because otherwise, $\v(p(b_1, \ldots, b_n)) = m$ which contradicts the fact that $p(b_1, \ldots, b_n) = 0$.
Thus, the image of $\In_\w(p)$ in the quotient space $F_{\v \succeq m} / F_{\v \succ m}$ is $0$, i.e. $\In_\w(p) \in I_\v$ as required.
\end{proof}

The next theorem gives necessary and sufficient conditions for a set $\B$ of algebra generators to be a Khovanskii basis.
\begin{theorem} \label{th-Khovanskii-basis-equiv-conditions}
Let $\B = \{b_1, \ldots, b_n\}$ be a set of algebra generators for $A$. The following conditions are equivalent.
\begin{itemize}
\item[(1)] $\B$ is a Khovanskii basis.
\item[(2)] The ideals $\In_\w(I)$ and $I_\v$ coincide.
\item[(3)] Let $\{p_1, \ldots, p_s\}$ be $M$-homogeneous generators for the ideal $I_\v$. Then, for $i=1, \ldots, s$, 
the subduction algorithm (Algorithm \ref{algo-subduction}) is applicable to represent \linebreak $p_i(b_1, \ldots, b_n)$ as a polynomial in the $b_i$.
\end{itemize}
\end{theorem}
\begin{proof}
Recall that for any $0 \neq f \in A$ we let $\bar{f}$ denote its image in $\gr_\v(A)$.
(1) $\Rightarrow$ (2). Let $p(\x) = \sum_\alpha c_\alpha \x^\alpha \in I_\v$ be an $M$-homogeneous polynomial. As before let $m(p) = \MIN\{ M\alpha \mid c_\alpha \neq 0\}$. Also let $a = \v(p(b_1, \ldots, b_n))$. We know that $p(\bar{b}_1, \ldots, \bar{b}_n) = 0$. This implies that $a \succ m(p)$.
Since $\B$ is assumed to be a Khovanskii basis, as in the proof of Lemma \ref{lem-Khov-basis-value-semigp}, we can find a polynomial $p_1(\x) = \sum_\beta c'_\beta \x^\beta$ such that 
$\overline{p(b_1, \ldots, b_n)} = p_1(\bar{b}_1, \ldots, \bar{b}_n)$
and moreover for every monomial $c_\beta \x^\beta$ appearing in $p_1$ we have $M\beta = a$. Continuing with the subduction algorithm applied to $p(b_1, \ldots, b_n)$ (Algorithm \ref{algo-subduction}) we obtain a polynomial 
$q(\x) = p_1(\x) + q_1(\x)$ such that $p(b_1, \ldots, b_n) = q(b_1, \ldots, b_n)$ and $\In_\w(q) = \In_\w(p_1)$. It follows that $p - q \in I$ and also $\In_\w(p - q) = p$. This shows that $p \in \In_\w(I)$ as required.

(2) $\Rightarrow$ (1). Let $A'$ denote the subalgebra of $\gr_\v(A)$ generated by the $\bar{b}_i$. Suppose by contradiction that $\In_\w(I) = I_\v$ but $\B$ is not a Khovanskii basis. Then there exists $p(\x) = \sum_\alpha c_\alpha \x^\alpha \in \k[\x]$ such that 
\begin{equation} \label{equ-v-notin-S'}
\overline{p(b_1, \ldots, b_n)} \notin A'. 
\end{equation}
Note that $m(p) = \MIN\{M \alpha \mid c_\alpha \neq 0\}$ is a nonnegative integer linear combination of the $\v(b_i)$ and hence $m(p) \in S'$, the semigroup generated by the $\v(b_i)$. 
By assumption, the value semigroup $S$, and hence its subsemigroup $S'$, are 
maximum well-ordered. Thus, without loss of generality, we can assume that $m(p)$ is maximum among all the polynomials 
satisfying \eqref{equ-v-notin-S'}. For \eqref{equ-v-notin-S'} to hold, we must have 
$\v(\In_\w(p)(b_1, \ldots, b_n)) \succ m(p)$ which shows that $p \in I_\v$. From the equality of $I_\v$ and $\In_\w(I)$ we then conclude that there exists 
$q \in I$ such that $\In_\w(q) = \In_\w(p)$.  Since $q \in I$ we see that $(p-q)(b_1, \ldots, b_n) = p(b_1, \ldots, b_n)$ and hence 
$\overline{(p-q)(b_1, \ldots, b_n)} = \overline{p(b_1, \ldots, b_n)} \notin A'$. On the other hand, $\In_\w(q) = \In_\w(p)$ implies that $m(p-q) \succ m(p)$.
This contradicts that $m(p)$ was maximum among the polynomials satisfying \eqref{equ-v-notin-S'}. This finishes the proof.
 
(1) $\Rightarrow$ (3) follows from definitions, we only need to prove (3) $\Rightarrow$ (1). By Lemma \ref{lem-in_w-I_B} and (2) above it is enough to show that $I_\v \subset \In_M(I)$. Let $1 \leq i \leq n$. Since $p_i(\bar{b}_1, \ldots, \bar{b}_n) = 0$ we know that $\v(p_i(b_1, \ldots, b_n))$ is strictly greater than $m(p_i)$. By assumption, the subduction algorithm (Algorithm \ref{algo-subduction}) produces a polynomial $q_i(\x)$ such that $p_i(b_1, \ldots, b_n) = q_i(b_1, \ldots, b_n)$ and $m(q_i) \succ m(p_i)$.  
Thus, $p_i - q_i \in I$ and $p_i = \In_\w(p_i - q_i) \in \In_\w(I)$. It follows that $I_\v \subset \In_\w(I)$. \end{proof}

We can now present an algorithm to find a finite Khovanskii basis starting from a set of algebra generators, provided that such a basis exists. 
\begin{algorithm}[Finding a finite Khovanskii basis] \label{algo-find-Khovanskii-basis}
{\bf Input:} A finite set of $\k$-algebra generators $\{b_1, \ldots, b_n\}$ for $A$. {\bf Output:} A finite Khovanskii basis $\mathcal{B}$.
\begin{itemize}
\item[(0)] Put $\B = \{b_1, \ldots, b_n\}$. Let $\bar{\B}$ be the image of $\B$ in $\gr_\v(A)$.
\item[(1)] Let $I_\v$ be the kernel of
homomorphism $\k[x_1, \ldots, x_n] \to \gr_\v(A)$.  
Let $G$ be a finite set of generators for $I_\v$.
\item[(2)] Take an element $g \in G$. 
Let $h \in A$ be the element obtained by plugging $b_i$ for $x_i$ in $g$, $i=1, \ldots, n$. Let $\bar{h}$ denote the image of $h$ in $\gr_\v(A)$. 
\item[(3)] Verify if $\bar{h}$ lies in the subalgebra generated by $\bar{\B}$.
\item[(4)] If this is the case, find a polynomial $p(x_1, \ldots, x_n)$ such that $\bar{h} = p(\bar{b}_1, \ldots, \bar{b}_n)$. This means that either $h = p(b_1, \ldots, b_n)$ or $\v(h - p(b_1, \ldots, b_n)) \succ \v(h)$. Put $h_1 = h - p(b_1, \ldots, b_n)$. If $h_1 = 0$ go to the step (6). Otherwise, replace $h$ with $h_1$ and go to the step (3).
\item[(5)] If $\bar{h}$ does not lie in the subalgebra generated by $\bar{\B}$ then add $h$ to $\B$. 
\item[(6)] Repeat until there are no generators left in $G$. 
\item[(7)] If no elements where added to $G$ then $\B$ is our desired finite Khovanskii basis. Otherwise go to step (1). 
\end{itemize}
\end{algorithm}

\begin{corollary} \label{cor-find-Khovanskii-basis-terminates}
Algorithm \ref{algo-find-Khovanskii-basis} terminates in a finite number of steps if and only if $(A, \v)$ has a finite Khovanskii basis.
\end{corollary} 
\begin{proof}
It follows from Theorem \ref{th-Khovanskii-basis-equiv-conditions} that if Algorithm  \ref{algo-find-Khovanskii-basis} terminates then $\B$ is a Khovanskii basis for $(A, \v)$. Now suppose $(A, \v)$ has a finite Khovanskii basis. We would like to show that the algorithm terminates. After $i$-th iteration of the algorithm (step (7)) we obtain an algebra generating set $\B_i$ for $A$ with $\B=\B_0 \subset \B_1 \subset \B_2 \subset \cdots$. Let $\hat{\B} = \bigcup_i \B_i$. We note that Theorem \ref{th-Khovanskii-basis-equiv-conditions}, with the same proof, still holds for a possibly infinite algebra generating set. Thus because of the way the sets $\B_i$ are constructed, Theorem \ref{th-Khovanskii-basis-equiv-conditions} can be applied to $\hat{\B}$ to conclude that it is a Khovanskii basis. Now since by assumption $(A, \v)$ has a finite Khovanskii basis, it is easy to see that $\hat{\B}$ contains a finite Khovanskii basis and hence the algorithm must have terminated in finite time.
\end{proof}

\subsection{Background on Newton-Okounkov bodies}   \label{subsec-NO-bodies}
Finally we briefly discuss the definition and main properties of a Newton-Okounkov body associated to a positively graded algebra $A$. It is a convex body which encodes information about the asymptotic behavior of Hilbert function of $A$. It is a far generalization of the Newton polytope of a 
projective toric variety. Our presentation here is close to the approach in \cite{KK}. 

We begin with the definition of a Newton-Okounkov cone.
\begin{definition}[Newton-Okounkov cone]  \label{def-Newton-Okounkov-cone}
Let $A$ be a (not necessarily graded) domain. Let $\v: A \setminus\{0\} \to \Z^r$ be a valuation. 
We define the {\it Newton-Okounkov cone} $P(A, \v)$ to be the closure of the convex hull of $S$, where $S = S(A, \v)$ is the value semigroup. Note that $0 \in S$ because by assumption $\v$ has value $0$ on $\k$. 
\end{definition}
We note that if $S$ is a finitely generated semigroup then the cone $P(A, \v)$ is a rational polyhedral cone, but the converse is not true (see for example \cite[Example 6]{Anderson}).

Now we follow \cite[Section 2.3]{KK} and take $A = \bigoplus_{i \geq 0} A_i$ to be a positively graded algebra and domain. Without loss of generality we can assume  that $A$ is embedded, as a graded $\k$-algebra, into a polynomial ring $F[t]$ (in one indeterminate $t$) where $F$ is a field containing $\k$. For example one can take $F$ to be the degree $0$ part of the quotient field of $A$. Let $\v': F\setminus \{0\} \to \Z^r$ be a valuation.  We can extend $\v'$ to a valuation  $\v: A \setminus \{0\} \to \N \times \Z^r$ which refines the grading by degree as follows. Firstly, equip $\Z^{r+1}$ with the following group ordering: for $(m, a), (n, b) \in \Z \times \Z^r$, let us say that $(m, a) \succ (n, b)$ if either $m < n$, or $m=n$ and $a \succ b$. Now let $f \in A$ be an element of degree $m$ and write $f = \sum_{i=0}^m f_i$ as sum of its homogeneous components. We put $\v(f) = (m, \v'(f_m))$. One verifies that $\v$ is a valuation and moreover, if $\v'$ has one-dimensional leaves then $\v$ also has one-dimensional leaves. 

\begin{definition}[Newton-Okounkov body]  \label{def-NO-body}
Let $(A, \v)$ be as above. The {\it Newton-Okounkov body $\Delta(A, \v)$} is defined to be the intersection of the Newton-Okounkov cone $P(A, \v)$ with 
the plane $\{1\} \times \R^r$. Alternatively, $\Delta(A, \v)$ can be defined as:
$$\Delta(A, \v) = \overline{\conv(\bigcup_{i>0} \{\v'(f)/i \mid 0 \neq f \in A_i \})} \subset \R^r.$$
\end{definition}

\begin{remark} \label{rem-NO-body-bounded}
Note that in the definition we do not require that $A$ is a finitely generated algebra. 
Without any assumption on $A$ the corresponding set 
$\Delta(A, \v)$ may be unbounded and not interesting. One shows that if $A$ is contained in a finitely generated graded algebra (in particular if $A$ itself is finitely generated) 
then the corresponding $\Delta(A, \v)$ is bounded and hence is a convex body. 
\end{remark}

The following is the main result about the Newton-Okounkov bodies of graded algebras. Let $A$ be a positively graded algebra. As above equip $A$ with a valuation $\v: A \setminus \{0\} \to \N \times \Z^r$. Recall that the Hilbert function of $A$ is the function $H_A: \N \to \N$ defined by $H_A(i) = \dim_\k(A_i)$, for all $i$.
\begin{theorem}  \label{th-NO-bodies-main}
Let us assume that $A$ is contained in a finitely generated algebra. Also assume that the valuation $\v$ has one-dimensional leaves. We then have
$$\lim_{i \to \infty} \frac{H_A(i)}{i^q} = \vol_q(\Delta(A, \v)),$$
where $q$ is the dimension of the Newton-Okounkov body $\Delta(A, \v)$ and $\vol$ denotes the (appropriately normalized) $q$-dimensional volume
in the affine span of $\Delta(A, \v)$.
\end{theorem}

\begin{corollary}  \label{cor-NO-bodies-degree-proj-var}
Let $Y$ be a projective variety of dimension $d$ sitting in a projective space $\mathbb{P}^N$. Let $A$ be the homogeneous coordinate ring of $Y$. Equip $A$ with a valuation $\v$ with one-dimensional leaves as above. Then the degree of $Y$ is equal to $d!$ times the volume of the convex body 
$\Delta(A, \v) \subset \R^d$.
\end{corollary}

\begin{remark}  \label{rem-toric-degen}
When $(A, \v)$ has a finite Khovanskii basis, the corresponding Newton-Okounkov body $\Delta(A, \v)$ is a rational polytope and we have a toric degeneration of $Y = \Proj(A)$ to a (not-necessarily normal) toric variety whose normalization is the toric variety associated to $\Delta(A, \v)$ (\cite{Anderson}, \cite[Section 7]{Kaveh-Crystal} and \cite{Teissier}).
\end{remark}

\subsection{Quasivaluations and filtrations}    \label{subsec-quasival-filtration}
It is conceptually useful to relax the valuation axioms and consider the so-called quasivaluations.
A quasivaluation differs from a valuation in that it is only superadditive with respect to multiplication. 

\begin{definition}     \label{def-quasival}
Let $(\Gamma, \succ)$ be a linearly ordered abelian group and let $A$ be a $\k$-algebra. A function $\v: A \setminus \{0\} \to \Gamma$ is said to be  a {\it quasivaluation} over $\k$ if the following properties hold:
\begin{enumerate}
\item For all $0 \neq f, g, f+g$ we have $\v(f + g) \succeq \MIN\{\v(f), \v(g)\}$.
\item For all $0 \neq f, g \in A$ we have $\v(fg) \succeq \v(f) + \v(g)$.
\item For all $0 \neq f \in A$ and $0 \neq c \in \k$ we have $\v(cf) = \v(f)$.
\end{enumerate}
It is sometimes useful to define a quasivaluation to be a map $\v: A \to \Gamma \cup \{\infty\}$ satisfying the above axioms, where $\infty$ is great than all elements in $\Gamma$.
\end{definition}

For the cases we consider $\Gamma$ will be $\Q^r$ with a linear ordering and $\v$ will be assumed to be discrete, i.e. its image is a discrete subset of $\Q^r$.  Similar to valuations, a quasivaluation $\v$ defines a corresponding filtration $\mathcal{F}_{\v} = \{F_{\v \succeq a} \mid a \in \Q^r\}$ on $A$. A {\it quasivaluation with one-dimensional leaves} is defined as before, namely we require that for each $a \in \Q^r$ the quotient space $F_{\v \succeq a} / F_{\v \succ a}$ is at most $1$-dimensional (see Definition \ref{def-valuation}(4)). \footnote{{As pointed out to us by Peter Littelmann, contrary to Theorem \ref{th-full-rank-1-dim-leaves} for valuations, one can find examples of full rank quasivaluations on an algebra over an algebraically closed filed $\k$ that do not have one-dimensional leaves.}}

Conversely, let $\mathcal{F} = \{F_a\}_{a \in \Q^r}$ be a decreasing algebra filtration of $A$ by $\k$-vector subspaces such that for any $0 \neq f \in A$ there exists $a \in \Q^r$ such that $f \in F_a \setminus \bigcup_{a' \succ a} F_{a'}$.  Then the function $\v_\F: A \setminus \{0\} \to \Q^r$ defined by:
\begin{equation}  \label{qf}
\v_\F(f) = \MAX\{a \in \Q^r \mid f \in F_a\},
\end{equation}
is a quasivaluation.
The two constructions of $\F_\v$ and $\v_\F$ are inverse to each other when $\v$ is discrete. For any filtration $\mathcal{F} = \{ F_a\}_{a \in \Q^r}$, one defines the associated graded algebra $\gr_{\mathcal{F}}(A)$ by
\begin{equation}
\gr_{\mathcal{F}}(A) = \bigoplus_{a \in \Q^r} F_a/F_{\succ a},
\end{equation}
where $F_{\succ a} = \bigcup_{a' \succ a} F_{a'}$. When $\mathcal{F} = \mathcal{F}_{\v}$ for some quasivaluation $\v$ we write $\gr_{\v}(A)$ instead of  
$\gr_\mathcal{F}(A)$. A discrete quasivaluation $\v$ is a valuation if and only if $\gr_{\v}(A)$ is a domain.  
A special case of the construction $\v_\F$ is the valuation associated to a grading in Example \ref{ex-valuation}(1).

\subsection{Adapted bases}
In this section we introduce the vector space counterpart of a Khovanskii basis.

\begin{definition}   \label{def-adapted-basis}
A $\k$-vector space basis $\mathbb{B} \subset A$ is said to be {\it adapted to a filtration $\mathcal{F} = \{F_a\}_{a \in \Q^r}$} if $F_a \cap \mathbb{B}$ is a vector space basis for $F_a$, for all $a$.  Similarly $\mathbb{B}$ is said to be adapted to a quasivaluation $\v$ if it is adapted to its associated filtration $\mathcal{F}_{\v}$. 
\end{definition}

{We would like to point out that when $\v$ has an adapted basis then the maximum in \eqref{qf} is always attained.}

\begin{example}
As in Example \ref{ex-valuation}(1) let $A = \bigoplus_{g \in \Gamma} A_g$ be a $\Gamma$-grading of an algebra $A$ where $\Gamma$ is an ordered group. For each $g \in \Gamma$ let $\mathbb{B}_g$ be a $\k$-vector space basis for $A_g$ and let $\mathbb{B} = \bigcup_{g \in \Gamma} \mathbb{B}_g$. It is straightforward to see that $\mathbb{B}$ is adapted to the valuation $\v$ associated to the $\Gamma$-grading. An important special case of this is considered in Section \ref{subsec-weight-quasival} where the set of monomials is an adapted basis for a polynomial algebra $\k[\x]$ with respect to any weight valuation. 
\end{example}

\begin{example}  \label{ex-adapted-basis-rep-theory}
Let $G$ be a connected reductive group over an algebraically closed characteristic $0$ field $\k$, and let $U \subset G$ be a maximal unipotent subgroup. As a $G$-module, the coordinate ring $\k[G/U]$ of the variety $G/U$ is known to decompose into a direct sum $\bigoplus_{\lambda \in \Lambda_+} V(\lambda)$ over all irreducible representations of $G$.  Each of these representations has a distinguished (dual) canonical basis $\mathbb{B}(\lambda) \subset V(\lambda)$ constructed by Lusztig (\cite{Lusztig}). The set $\mathbb{B} = \coprod_{\lambda \in \Lambda_+} \mathbb{B}(\lambda)$ is the dual canonical basis of $\k[G/U]$.  

For each reduced decomposition $\S$ of the longest word 
$w_0$ of the Weyl group of $G$ there is a valuation $\v_{\S}$ on the coordinate ring of $G/U$ which has one-dimensional leaves and is adapted to $\mathbb{B}$ (see \cite{Kaveh-Crystal, M-NOK}).  These are known as string valuations; they provide a method to construct toric degenerations of $G/U$ as well as any flag variety of $G$ (\cite{Caldero, AB, Kaveh-Crystal}).  

Other variants of adapted bases in representation theory are studied in greater generality by Feigin, Fourier, and Littelmann in \cite{FFL}, where they are called \emph{essential bases}. 
\end{example}
 
\begin{remark} \label{rem-adapted-basis-1dimleaves}
It immediately follows from the definition that the set of values of $\v$ on $A$ coincides with the set of values of $\v$ on any adapted basis $\mathbb{B}$. Moreover, if $\v$ has one-dimensional leaves then a subset $\mathbb{B}$ is an adapted basis if and only if $b \mapsto \v(b)$ gives a bijection between $\mathbb{B}$ and the set of values of $\v$.  
\end{remark}
We can formulate a vector space version of the subduction algorithm (Algorithm \ref{algo-subduction}). Let $\bar{\mathbb{B}} \subset \gr_{\v}(A)$ be a vector space basis consisting of homogeneous elements. Also let $\mathbb{B} \subset A$ be a lift of $\bar{\mathbb{B}}$ to $A$, i.e. for each $\bar{b} \in \bar{\mathbb{B}}$, we have a unique $b \in\mathbb{B}$ whose image is $\bar{b}$.

\begin{algorithm}[Vector space subduction] \label{algo-vec-space-subduct}
{\bf Input:} A vector space basis $\bar{\mathbb{B}} \subset \gr_{\v}(A)$, a lift $\mathbb{B} \subset A$ of $\bar{\mathbb{B}}$ and an 
element $f \in A$. {\bf Output:} An expression of $f$ as a linear combination of the elements in $\mathbb{B}$.
\begin{itemize}
\item[(1)] Compute $\v(f) = a$ and take the equivalence class $\bar{f} \in F_{\v \succeq a} / F_{\v \succ a}$.
\item[(2)] Express $\bar{f}$ as a linear combination of elements in $\bar{\mathbb{B}}$, that is, $\bar{f} = \sum_i c_i\bar{b}_i$.
\item[(3)] If $f = \sum_i c_ib_i$ we are done. Otherwise replace $f$ with $f - \sum_i c_i b_i \in F_{\v \succ a}$ and go to (1). 
\end{itemize}
\end{algorithm}

We have the following lemma. We omit the straightforward proof. 
\begin{lemma}\label{subductadapt}
A lift $\mathbb{B} \subset A$ of a basis $\bar{\mathbb{B}} \subset \gr_{\v}(A)$
is a vector space basis for $A$ (and hence a basis adapted to $\v$) if and only if Algorithm \ref{algo-vec-space-subduct} terminates for all $f \in A$ after a finite number of steps. In this case, we have the following: for any $0 \neq f \in A$ write $f = \sum_i c_ib_i$ as a linear combination of the basis elements $b_i \in \mathbb{B}$. Then $\v(f) = \MIN\{\v(b_i) \mid c_i \neq 0\}$. 
\end{lemma}

Many different vector space bases of $A$ can be adapted to the same quasivaluation $\v$. Any two such bases are related by a lower triangular change of coordinates. 
\begin{proposition}\label{lowertriangular}
Let $\v$ be a quasivaluation with one-dimensional leaves.
Let $\mathbb{B}$, $\mathbb{B}' \subset A$ be adapted to $\v$. Then every $b \in \mathbb{B}$ has a lower-triangular expression in the basis $\mathbb{B}'$, and vice versa:
$$b = cb' + \sum_{\v(b_i') \succ \v(b)} c_ib_i', \quad \v(b) = \v(b'),$$
with $c$ and the $c_i \in \k$ and $c \neq 0$.
\end{proposition}
\begin{proof}
This follows from Lemma \ref{subductadapt}. 
\end{proof}

\section{Valuation constructed from a weighting matrix}\label{sec-weightandsubductive}
In this section we introduce two classes of quasivaluations on an algebra $A$. First is the class of weight quasivaluations (Definition \ref{def-weight-quasival}). These are quasivaluations which are induced from a vector-valued weighting of indeterminates in a polynomial algebra $\k[\x]$ which presents $A$. When the weighting matrix lies in the Gr\"obner region, the corresponding weight quasivaluation possesses an adapted vector space basis (in the sense of Definition \ref{def-adapted-basis}).
We also describe the set of weight quasivaluations on $A$ as a piecewise linear object (Section \ref{subsec-set-weight-quasival}). 
The second class is what we call subductive valuations (Definition \ref{def-subductive-val}). These are valuations that have a finite Khovanskii basis and for which the subduction algorithm (Algorithm \ref{algo-subduction}) always terminates. One of the important results in this section is that every subductive valuation is a weight valuation (Section \ref{lem-subductiveweight}). 

\subsection{Quasivaluation constructed from a weighting matrix}  \label{subsec-weight-quasival}
We start by introducing the notions of filtration and quasivaluation constructed out of a weighting matrix (in fact, we already saw these notions in disguise in Section \ref{subsec-Khovanskii-basis} after Remark \ref{rem-I_A-toric-ideal}). Let $\pi: B \to A$ be a surjection of $\k$-algebras, and let $\mathcal{F} = \{F_a\}$ be an algebra filtration on $B$ by $\k$-vector spaces. The pushforward filtration $\pi_*(\mathcal{F})$ on $A$ is defined by the set of spaces  $\{\pi(F_a)\}$.   If $\v$ is a quasivaluation on $B$ with corresponding filtration $\mathcal{F}_\v$, we let $\pi_*(\v)$ be the pushforward quasivaluation on $A$ corresponding to the filtration $\pi_*(\mathcal{F}_{\v})$.   

Fix a group ordering $\succ$ on $\Q^r$.
Each matrix $\w \in \Q^{r \times n}$ defines a $\Q^r$-valued valuation $\tilde{\v}_\w: \k[\x] \setminus \{0\} \to \Q^r$ by the following rule. Let $p = \sum_\alpha c_\alpha \x^\alpha \in \k[\x]$. Define:
\begin{equation}   \label{equ-tilde-v}
\tilde{\v}_\w(p) = \MIN\{ M\alpha \mid c_{\alpha} \neq 0\}.
\end{equation}
Here $\MIN$ is computed with respect to $\succ$.
We denote the filtration on $\k[\x]$ corresponding to $\tilde{\v}_\w$ by $\mathcal{F}_\w$. Notice that the monomial basis of $\k[\x]$ is adapted to the filtration $\mathcal{F}_\w$, in particular $F_{\w, \succeq a}$ is the span of monomials 
$\x^{\alpha}$ with $M\alpha \succeq a$. 


\begin{definition}  \label{def-weight-quasival}
With notation as above, the {\it weight filtration} on 
$A$ associated to $\w \in \Q^{r \times n}$ is the pushforward filtration $\pi_*(\mathcal{F}_\w)$. We denote the corresponding quasivaluation on $A$ by $\v_\w$. We refer to $\v_\w$ as the {\it weight quasivaluation with weighting matrix $\w$}.
\end{definition}

\begin{lemma}   \label{lem-v_w-max-min}
For any $f \in A$ and $\w \in \Q^{r \times n}$, the quasivaluation $\v_\w(f)$ is computed as follows:
\begin{equation}   \label{equ-v_w}
\v_\w(f) = \pi_*(\tilde{\v}_\w)(f) = \MAX\{ \tilde{\v}_\w(\tilde{f}) \mid \tilde{f} \in \k[\x], ~ \pi(\tilde{f}) = f\}.
\end{equation}
Note that, as $\tilde{\v}_\w$ is defined by a minimum, the equation
\eqref{equ-v_w} is in fact a max-min formula. 
\end{lemma}

Throughout the rest of the paper, we assume that the weighting matrix $M$ is chosen such that the maximum in \eqref{equ-v_w} is attained for all $0 \neq f \in A$. This is the case for example if $M$ is chosen from the Gr\"obner region (this follows from Proposition \ref{prop-st-monomial-adapted-basis} below) or from the rank $r$ tropical variety $\mathcal{T}^r(I)$ (see Proposition \ref{trop-max-obtained}). In the case that $M \in \textup{GR}^r(I) \subset \Q^{r \times n}$, the weight quasivaluation $\v_M$ can be computed using a standard monomial basis as follows.


\begin{proposition} \label{prop-st-monomial-adapted-basis}
With notation as above, let $\w \in \textup{GR}^r(I)$ and let $\mathbb{B} \subset A$ be the standard monomial basis for a monomial ordering $>$ with $\w \in C_>(I)$. Then $\mathbb{B}$ is adapted to $\v_\w$. 


\end{proposition}
\begin{proof} 
The inequality $\v_\w(f) \succeq \MIN\{\v_\w(b_\alpha) \mid c_\alpha \neq 0\}$  is immediate from the definition of a quasivaluation (Definition 
\ref{def-quasival}(1)). This implies $\v_\w(f) \succeq \MIN\{M\alpha \mid c_\alpha \neq 0\}$. We need to show that the equality holds. Let $\tilde{f} = \sum_\alpha c_\alpha \x^\alpha$ and let $m = \MIN\{M\alpha \mid c_\alpha \neq 0 \}$. Suppose by contradiction that there is $\tilde{h} = \sum_\beta c'_\beta \x^\beta \in \k[\x]$ such that $\pi(\tilde{h}) = f$ and moreover for every $\beta$ with $c'_\beta \neq 0$ we have $M\beta \succ m$. Let $p = \sum_\alpha c_\alpha \x^\alpha - \sum_\beta c'_\beta \x^\beta$. Then $p \in I$ and $\In_\w(p)$ consists only of standard monomials $c_\alpha \x^\alpha$, this is a contradiction.
\end{proof}

It follows from Lemma \ref{lem-negativeGregion} and Proposition \ref{prop-st-monomial-adapted-basis} that if $I$ is a homogeneous ideal with respect to a positive grading on $\k[\x]$ then any weight quasivaluation $\v_\w$ can be equipped with an adapted basis. From now on we denote the associated graded algebra of the weight quasivaluation $\v_\w$ by $\gr_\w(A)$. 
The following lemma describes the graded algebra $\gr_{\w}(A)$ in terms of the initial ideal $\In_\w(I)$ of $I \subset \k[\x]$. 


\begin{lemma}\label{lem-gr_v-in_w}
The associated graded algebra $\gr_\w(A)$ is isomorphic to $\k[\x]/\In_\w(I)$. 
\end{lemma}
\begin{proof}
Consider the filtration on $\k[\x]$ by the spaces $F_{\w, \succeq a}$ and the associated pushforward filtration $\pi(F_{\w, \succeq a})$. For any $a \in \Q^r$, the pushforward space $\pi(F_{\w, \succeq a})$ can be identified 
with $F_{\w, \succeq a}/ (F_{\w, \succeq a} \cap I)$.  As such, the associated graded algebra $\gr_\w(A)$ is a direct sum of the following $\k$-vector spaces:
$$(F_{\w, \succeq a}/ F_{\w, \succeq a} \cap I)/(F_{\w, \succ a}/ F_{\w, \succ a} \cap I).$$
Since $\In_\w(I)$ is homogeneous with respect to the $\w$-grading on $\k[\x]$, we can also think of $\k[\x]/\In_\w(I)$ as a $\Q^r$-graded algebra.   In particular, $\k[\x]$ is canonically isomorphic to the associated graded algebra $\gr_\w(\k[\x])$ $= \bigoplus_a F_{\w, \succeq a}/F_{\w, \succ a}$, where $F_{\w, \succeq a}/F_{\w, \succ a}$ is the vector space spanned by the images of monomials with $M$-degree $a$.  The image of $\In_\w(I)$ under this isomorphism is the direct sum of the spaces  $(F_{\w, \succeq a} \cap I)/(F_{\w, \succ a} \cap I)$.   Now the lemma follows from the following general fact about quotients from linear algebra. Let $W, U$ be subspaces of a vector space $V$, then:
$$(V/U)/(W / W \cap U) \cong (V/W)/(U / U \cap W) \cong V/ (W + U).$$ 
\end{proof}


Here is a simple example for illustration.
\begin{example} \label{ex-weight-filt}
Let $A = \k[x]$ be the polynomial algebra in one indeterminate $x$ and consider its presentation $\k[x] \cong \k[x,y] / I$ where $I = \langle x^2 - y \rangle$. Thus we have the surjective homomorphism $\pi: \k[x, y] \to \k[x]$ given by $\pi(x)=x$ and $\pi(y) = x^2$. With notation as above, let $r=1$ and consider the weight $\w = (1, 2) \in \Q^2$. Then $I$ is a homogeneous ideal with respect to the $\w$-grading. It is easy to verify that $\In_\w(x^2 - y) = x^2 - y$ and hence $\In_\w(I) = I$. Also the pushforward filtration on $A=\k[x]$ is just the grading by degree and thus $\gr_\w(\k[x]) = \k[x] \cong \k[x,y] / I$ as expected.

Next, let $\w = (1, 3)$. In this case, one can compute the pushforward filtration on $\k[x]$ as follows. For each $a \geq 0$ we have 
$\pi(F_{\w, \geq a}) = \span\{ x^m, x^{m+1}, \ldots \}$ where $m = \lceil 2a / 3 \rceil$. 
It follows that the $\gr_\w(\k[x])$ is the graded algebra whose $a$-th graded piece is $\k$ when $a \equiv 0, 1 \pmod{3}$ and is $0$ when $a \equiv 2 \pmod{3}$. One verifies that the quotient $\k[x, y] / \langle x^2 \rangle$ is indeed isomorphic to this algebra. The isomorphism is given by sending the image of $x$ (in $\k[x, y]/ \langle x^2 \rangle$) to a nonzero element in degree $1$ (in $\gr_\w(\k[x])$) and sending the image of $y$ to a nonzero element in degree $3$. We remark that since the initial ideal $\In_\w(I) = \langle x^2 \rangle$ is not prime, and thus the associated graded algebra $\gr_\w(\k[x])$ is not a domain, the quasivaluation $\v_\w$ is not a valuation. 
\end{example}

It may be that the Gr\"obner region $\textup{GR}^r(I)$ of an ideal presenting $A$ is not all of $\Q^{r\times n}$; we show that in this case the quasi-valuation $\v_M$ will still take finite values on $A \setminus \{0\}$ provided $M$ is chosen from $\mathcal{T}^r(I)$.  

\begin{proposition}\label{trop-max-obtained}
Let $I$ be prime and $M \in \mathcal{T}^r(I)$, then for every $0 \neq f \in \k[\x]/I$, $\v_M(f) < \infty$. 
\end{proposition}

\begin{proof}
 For $M \in \Q^{r \times n}$, let $I_M$ be the set of $f \in \k[\x]$ such that for every $a \in \Q^r$ there exists a $g \in I$ such that $\tilde{\v}_M(f + g) > a$. It is straightforward
 to check that $I_M$ is an ideal containing $I$.  Furthermore, $I_M$ is strictly larger than $I$ if and only if $\v_M(f) = \infty$ for some $0 \neq f \in A$.  First we show that if $M \in \mathcal{T}^r(I)$, we must also have $M \in \mathcal{T}^r(I_M)$.  Suppose $f \in I_M$ and $\In_M(f) = C_{\alpha}\x^\alpha$, then it follows that $\tilde{\v}_M(f) = \tilde{\v}_M(C_{\alpha}\x^\alpha) = a$.  We must have $g \in I$ such that $\tilde{\v}_M(f + g) > a$, but for this to be the case we must have $\tilde{\v}_M(g) = a$ and $\In_M(g) = C_{\alpha}\x^\alpha$ which contradicts the fact that $M \in \mathcal{T}^r(I)$.  
 
 Now we consider the filtration $\mathcal{F}_M$ on $\k[\x]$ and its pushforwards $F$ and $F'$ on $\k[\x]/I$ and $\k[\x]/I_M$, respectively.  Observe that $F_{\succeq a} = \mathcal{F}_{M, \succeq a}/ \mathcal{F}_{M, \succeq a}\cap I$ and $F'_{\succeq a} = \mathcal{F}_{M, \succeq a}/ \mathcal{F}_{M, \succeq a}\cap I_M$ so there is a natural surjection: $F_{\succeq a} \to F'_{\succeq a} \to 0$.  The kernel of this surjection is $\mathcal{F}_{M, \succeq a} \cap I_M / \mathcal{F}_{M, \succeq a} \cap I$; we claim that this is the quotient $I_M/I$.  To see this, let $f \in I_M$ with $\tilde{\v}_M(f) = a < b$, then there is a $g \in I$ with $\tilde{\v}_M(f + g) > b$.  It must follow that $\tilde{\v}_M(g) = a$ and that the equivalence class $\bar{f} \in \mathcal{F}_{M, \succeq a} \cap I_M / \mathcal{F}_{M, \succeq a} \cap I$  lies in the subspace $\mathcal{F}_{M, \succeq b} \cap I_M / \mathcal{F}_{M, \succeq b} \cap I$.  Since $b$ was arbitrary, we conclude that  $\mathcal{F}_{M, \succeq a} \cap I_M / \mathcal{F}_{M, \succeq a} \cap I = I_M/I$.  But this then implies that $F_{\succeq a}/ F_{\succ a} \cong F'_{\succeq a}/F'_{\succ a}$, and $\gr_F(\k[\x]/I) \cong \gr_{F'}(\k[\x]/I_M)$.  
 
 Now by Lemma \ref{lem-gr_v-in_w} we have that $\k[\x]/\In_M(I) \cong \gr_F(\k[\x]/I)$ and $\k[\x]/\In_M(I_M) \cong \gr_{F'}(\k[\x]/I_M)$.  As $M$ is in the tropical varieties of both $I$ and $I_M$, we conclude that $\k[\x]/I$ and $\k[\x]/I_M$ have the same Krull dimension.  But $I \subset I_M$ and $I$ is prime, so $I = I_M$. 
  
\end{proof}

\subsection{The set of weight quasivaluations}   \label{subsec-set-weight-quasival}
In this section we describe the set of weight quasivaluations on $A$ coming from a given presentation.  
Let $\mathcal{V}_\B$ denote the set of all weight quasivaluations $\v_\w$ on $A$ for $\w \in \Q^{r \times n}$. Define the function $\T_\B: \V_\B \to \Q^{r \times n}$ as follows. For each $\v_M \in \V_\B$ let:
$$\T_\B(\v_\w) = (\v_\w(b_1), \ldots, \v_\w(b_n)).$$
The value $\v_{\w}(b_i)$ is not necessarily the $i$-th column of $\w$. In fact, by Lemma \ref{lem-v_w-max-min}, for each $i$, $\v_\w(b_i)$ is given by the max-min formula: 
\begin{equation}  \label{equ-v_w-b_i}
\v_\w(b_i) = \MAX\{ \MIN \{M\alpha \mid c_\alpha \neq 0 \} \mid x_i - \sum_\alpha c_\alpha \x^\alpha \in I \}.
\end{equation}
We remark that the map $\T_\B$ is an extension of the usual tropicalization map in tropical geometry to the set of weight quasivaluations. We also define a {\it contraction map} $\iota: \Q^{r \times n} \to \Q^{r \times n}$ by:
$$\iota(\w) = \T_\B(\v_\w)  = (\v_\w(b_1), \ldots, \v_\w(b_n)),$$
for every $\w \in \Q^{r \times n}$. 
From \eqref{equ-v_w-b_i} we see that $\iota$ is a piecewise linear map. 

The purpose of this section is to prove the proposition below.
\begin{proposition}  \label{prop-set-quasival}
\begin{itemize} We have the following:
\item[(1)] $\v_\w = \v_{\iota(\w)}$, $\forall \w \in \Q^{r \times n}$. 
\item[(2)] $\iota(\iota(\w)) = \iota(\w)$, $\forall \w \in \Q^{r \times n}$.
\item[(3)] For $\w, \w' \in \Q^{r \times n}$, the equality $\v_\w = \v_{\w'}$ holds if and only if $\iota(\w) = \iota(\w')$. 
\item[(4)] If $\w$ is contained in the tropical variety $\T^r(I) \subset \Q^{r \times n}$, namely those weights for which $\In_\w(I)$ contains no monomial, then $\iota(\w) = \w$.
\end{itemize}
\end{proposition}

\begin{proof}
Let $\w \in \Q^{r \times n}$ and let $w_1, \ldots, w_n$ (respectively $w'_1, \ldots, w'_n$) denote the column vectors of $M$ (respectively the column vectors of $\iota(M)$). Then for all $1 \leq i \leq n$ we have $w'_i \succeq w_i$, and $w'_i \succ w_i$ if and only if $x_i \in \In_\w(I)$, this proves part (4) and shows that $\v_{\iota(\w)}(f) \succeq \v_\w(f)$ for any $f \in A$ and $F_{\w, \succeq a} \subseteq F_{\iota(\w), \succeq a}$ for all $a \in \Q^r$.   Furthermore, parts (2) and (3) are straightforward corollaries of (1).  To prove $(1)$, let $f \in F_{\iota(\w), \succeq a}$, and $\sum_\alpha c_\alpha \x^\alpha = p(\x) \in \k[\x]$ be a polynomial with $\pi(p(\x)) = f$ and $\iota(\w)\alpha \succeq a$ for all $c_\alpha \neq 0$.  Since $\v_\w$ is a quasivaluation, we must have $\v_\w(f) \succeq MIN\{\v_\w(\pi(\x^\alpha)) \mid c_\alpha \neq 0\}$ and  $\v_\w(\pi(x^\alpha)) \succeq (\v_\w(\pi(x_1)), \dots, \v_\w(\pi(x_n)))\alpha = \iota(\w)\alpha$. By the definition of $\iota(M)$ we get $\v_\w(f) \succeq MIN\{\iota(\w)\alpha \mid c_\alpha \neq 0\} \succeq a$, and $F_{\iota(\w), \succeq a} \subseteq F_{\w, \succeq a}$. 
\end{proof}

\subsection{Subductive valuations and proof of Lemma \ref{lem-intro-subductive-weight-val}}\label{lem-subductiveweight} 


A valuation $\v$ on an algebra $A$ with generating set $\B$ is a weight valuation for $M \in \Q^{r\times n}$ if and only if for every $f \in A$ there is a $p(\x) \in \k[\x]$ such that $p(b_1, \ldots, b_n) = f$ and $\In_M(p(\x)) = \sum C_{\alpha}\x^\alpha$ where $M\alpha = a = \v(f)$ for all $C_{\alpha} \neq 0$.  One way to ensure this condition holds is to take $\v$ to be a subductive valuation. 

\begin{definition}(Subductive valuation) \label{def-subductive-val}
A valuation $\v: A \setminus \{0\} \to \Q^r$ is said to be a {\it subductive valuation} if there is a finite Khovanskii basis $\mathcal{B} \subset A$ for $\v$ such that the subduction algorithm (Algorithm \ref{algo-subduction}) always terminates in finite time for any $f \in A$. 
\end{definition}


Now we give a proof of Lemma  \ref{lem-intro-subductive-weight-val} from the introduction. Recall the statements:
\begin{enumerate}
\item[(1)] $\v$ is a subductive valuation with respect to $\B \subset A$,
\item[(2)] $\v$ has an adapted basis $\mathbb{B}$ consisting of monomials in $\B$,
\item[(3)] $\v$ coincides with the weight valuation $\v_M$ for the matrix $M \in \Q^{r\times n}$ with column vectors $\v(b_1), \ldots, \v(b_n)$,
\item[(4)] $\v$ has Khovanskii basis $\B$.  
\end{enumerate} 
\noindent
We show these satisfy $(1) \Rightarrow (2) \Rightarrow (3) \Rightarrow (4)$.

\begin{proof}
$(1) \Rightarrow (2)$. Let $\bar{b} \in \gr_\v(A)$ be the equivalence class of $b \in \B$. For each $a$ we choose a maximal linearly independent set of monomials $\bar{b}^\alpha \in \gr_\v(A)$, where $b^\alpha$ has homogeneous degree $a$. Since $\bar{b}^\alpha$ are linearly independent, so are the $b^{\alpha}$. As a consequence, if $\v$ is subductive, any $f \in A$ can be written as a linear combination of these elements in a unique way.  
$(2) \Rightarrow (3)$. If $\mathbb{B} \subset A$ is an adapted basis of $\B$-monomials, then by assumption any $f \in A$ can be written as $p(b_1, \ldots, b_n)$ with the property that the highest monomials all have valuation equal to $\v(f)$. This implies that $\v = \v_M$.  Now $(3) \Rightarrow (4)$ because any weight valuation for $\B \subset A$ has finite Khovanskii basis $\B$. 
\end{proof}

By Example \ref{ex-graded-alg-max-well-ordered} we see that for any algebra $A$ graded by an abelian group $\Gamma$ we have that $(4) \Rightarrow (1)$ so that the above conditions are equivalent; this is also the case if the value semigroup $S(A, \v)$ is well-ordered.  In particular, (1)-(4) are equivalent if $A$ is positively graded and $\B \subset A$ consists of homogeneous elements (possibly of different degrees).

\section{Valuations from prime cones}\label{sec-val-from-prime-cone}
Recall that $A$ is a finitely generated algebra and domain and $\B$ is a finite set of algebra generators for $A$ giving rise to a presentation $A \cong \k[\x] / I$. 
This section concerns the proof of one of the main results of the paper (Theorem \ref{th-intro-main2} from the introduction). {First we describe the construction of a valuation on $A$ from a prime cone $C \subset \T(I)$ such that $\B$ is a finite Khovanskii basis for this valuation (see below for the definition of a prime cone). 
Moreover, we show that if $C \subset \GR(I)$ this valuation has an adapted basis (Definition \ref{def-adapted-basis}). When $A$ is positively graded and $\B$ consists of homogeneous elements, the valuation corresponding to $C$ is subductive (Definition \ref{def-subductive-val}).}

Let $I \subset \k[\x]$ be a prime ideal and let $C \subset \T(I)$ be an open cone in the tropical variety of $I$ such that for any $\u_1, \u_2 \in C$ we have $\In_{\u_1}(I) = \In_{\u_2}(I)$. For example, this is the case if $C$ is chosen from the Gr\"obner fan of the homogenization $I_h$ of $I$. Recall that this common initial ideal is denoted by $\In_C(I)$

\begin{definition}   \label{def-prime-cone}
Let $C \subset \mathcal{T}(I)$ be an open cone. We call $C$ a {\it prime cone} if the corresponding initial ideal $\In_C(I)$ is a prime ideal.\footnote{As mentioned in the introduction, by abuse of terminology, we may occasionally refer to a closed cone as prime, in which case we mean that its relative interior is prime.}
\end{definition}

Take a finite subset ${\bf u} = \{\u_1, \ldots, \u_r\} \subset C$. We denote the $r \times n$ matrix whose $j$-th row is $\u_j$ by $M$ and regard it as a $\Q^r$-weighting matrix on $\k[\x]$, where $\Q^r$ is given the standard lexicographic ordering. We denote the $i$-th column of $M$ by $w_i \in \Q^r$.  

\begin{proposition}\label{prop-val-from-cone}
{Let $C$ be a prime cone}, we have the following:
\begin{itemize}
\item[(1)] The weight quasivaluation $\v_\w$ is in fact a valuation with rank equal to $\rank(M)$.
\item[(2)] The associated graded algebra $\gr_\w(A)$ is isomorphic to $\k[\x] / \In_C(I)$.
\item[(3)] The value semigroup $S(A, \v_\w)$ is generated by the column vectors of $M$, which are in fact the vectors $\v_M(b_1)$, \ldots, $\v_M(b_n)$. Consequently, the Newton-Okounkov cone $P(A, \v_\w)$ is the cone generated by these column vectors.
\item[(4)] If the cone $C$ has maximal dimension $d = \dim(A)$ and 
the linear span of the set ${\bf u}$ is also $d$-dimensional then the valuation $\v_\w$ has rank $d$. If, in addition, we assume that $\k$ is algebraically closed then $\v_\w$ is a valuation with one-dimensional leaves. 
\end{itemize}
\end{proposition}

\begin{remark}  \label{rem-pos-graded-Grobner-region}
{We note that if $A$ is positively graded and we choose a set of homogeneous generators for $A$, any prime cone lies in the Gr\"obner region.}
\end{remark}

\begin{proof}[Proof of Proposition \ref{prop-val-from-cone}]
By Lemma \ref{lem-in_w-in_u_i} we have 
$\In_\w(I) = \In_{\u_r}( \cdots (\In_{\u_1}(I)) \cdots )$. Since by assumption $\In_{\u_1}(I) = \cdots = \In_{\u_r}(I) = \In_C(I)$ we conclude that $\In_\w(I) = \In_C(I)$ which is assumed to be a prime ideal. On the other hand, by Lemma \ref{lem-gr_v-in_w} we know that $\gr_{\v_\w}(A) \cong \k[\x] / \In_\w(I)$. Since the quotient $\k[\x]/\In_\w(I)$ is a domain we see from Proposition \ref{trop-max-obtained} that $\v_\w$ is indeed a valuation. Part (2) now follows from Part (1). To prove Part (3) we note that by (2) the image of $\B$ in $\gr_M(A)$ is an algebra generating set and hence $\B$ is a finite Khovanskii basis. Lemma \ref{lem-Khov-basis-value-semigp} then implies the claim. Part (4) follows from (1) and Theorem \ref{th-full-rank-1-dim-leaves}. 
\end{proof}

\begin{proposition}   \label{prop-v_w-v_u_i}
With notation as above, for $\u_i \in {\bf u}$ 
let $\v_{\u_i}: A \setminus \{0\} \to \Q$ denote the corresponding rank $1$ valuation. Then for any $b_j \in \B$ we have $\v_{\u_i}(b_j) = (\u_i)_j$, the $j$-th coordinate of the vector $\u_i \in \Q^n$.
\end{proposition}
\begin{proof}
Follows from Proposition \ref{prop-set-quasival}.
\end{proof}

\begin{remark}
(1) If we assume that $C$ is taken from the Gr\"obner fan of the homogenization $I_h$ of $I$, or if $I$ is itself homogeneous, then some of the $\u_i$ may be taken from faces of the closure of $C$, provided that the sum $\sum \u_i$ is in $C$. To show this, note that by Proposition \ref{midpoint} we can use the above argument (in the proof of Proposition \ref{prop-val-from-cone}) with $\u = \u_1 + \cdots + \u_r$ in place of $\u_1$.

(2) The proof of Proposition \ref{prop-val-from-cone} can also be used in the case that $A$ is the coordinate ring of a very affine variety, i.e. when $A$ is presented as a quotient of a Laurent polynomial algebra. 
\end{remark}

The next proposition shows that the value semigroup $S(A, \v_\w)$, up to linear isomorphism, depends only on the cone $C$. Before we state this result, let us define what we mean by linear isomorphism of subsets of vector spaces. Let $S \subset \Q^r$ and $S' \subset \Q^{r'}$ be two subsets. We say that $S$ is {\it linearly isomorphic} to $S'$ if there exists a $\Q$-linear map $T: \Q^r \to \Q^{r'}$ such that $T$ restricts to a bijection between $S$ and $S'$.

\begin{proposition}   \label{prop-S-lin-iso}
Suppose the span of ${\bf u}$ has dimension $\dim(C)$. Then, up to linear isomorphism, the semigroup $S(A, \v_\w)$ (and hence the cone $P(A, \v_\w)$) depends only on $C$.
\end{proposition}
\begin{proof}
Let ${\bf u} = \{\u_1, \ldots, \u_r\}$, ${\bf u}' = \{\u'_1, \ldots, \u'_s \} \subset C$ be two subsets with corresponding weighting matrices $\w \in \Q^{r \times n}$, $\w' \in \Q^{s \times  n}$ respectively. By assumption the spans of ${\bf u}$ and ${\bf u}'$ are the same. Hence every row of $M$ is a linear combination of the rows of $M'$. It follows that 
the $\w$-degree of two monomials $\x^\alpha$ and $\x^\beta$ are the same if and only if their $\w'$-degrees are the same. Now consider $R = \gr_\w(A)$ and $R' = \gr_{\w'}(A)$. The algebras $R$ and $R'$ are graded by $S = S(A, \v_\w)$ and $S' = S(A, \v_{\w'})$ respectively. By Proposition \ref{prop-val-from-cone}(2) we know that $R$ and $R'$ are isomorphic (because they are both isomorphic to $\k[\x] / \In_C(I)$). Moreover, by what we said above, the isomorphism sends each graded component $R_a$, $a \in S$ to another graded component $R_{a'}$, $a' \in S'$. One verifies that the map $a \mapsto a'$ gives a linear isomorphism between $S$ and $S'$.  
\end{proof}

Finally, let us assume that the algebra $A$ is positively graded, i.e. $A = \bigoplus_{i \geq 0} A_i$, and the algebra generating set $\B$ consists of homogeneous elements of degree $1$. It follows that $I$ is a homogeneous ideal and moreover the vector $(-1, \ldots, -1)$ belongs to the lineality space of $I$ and hence lies in every cone in the Gr\"obner fan of $I$. Thus, we can take the vector $\u_1 \in {\bf u}$ to be $(-1, \ldots, -1)$. In this case, one observes that the valuation $\v_{\w}$ constructed above is such that for every $0 \neq f \in A$, the first component of $\v_\w(f)$ is $-\deg(f)$ (after dropping the minus sign, the valuation $\v_M$ is of the form \eqref{equ-intro-v-degree}). 
In particular, $\v_\w$ is a subductive valuation (see Example \ref{ex-graded-alg-max-well-ordered}(1)). Moreover, Proposition \ref{primeconeadapted} below shows it has an adapted basis. The following is an immediate corollary of Proposition \ref{prop-val-from-cone}(3) and Proposition \ref{prop-S-lin-iso}.
\begin{corollary}   \label{cor-cone-NO-body-from-M_u}
With notation as above we have the following:
\begin{itemize}
\item[(1)] The Newton-Okounkov body $\Delta(A, \v_{\w})$ is the convex hull of the column vectors of $M$ (recall that the $i$-th column vector of $M$ coincides with $\v_{\w}(b_i)$).
\item[(2)] Up to linear isomorphism, the convex body $\Delta(A, \v_\w)$ depends only on $C$.
\end{itemize}
\end{corollary}




Back to the general case, where $A$ is not necessarily positively graded, when $C$ lies in $\GR(I)$ we can find a monomial ordering on $\k[\x]$ such that the cone $C$ is a face of a maximal cone $C_>$ in the Gr\"obner region $\GR(I)$. Let $\mathbb{B} \subset \k[\x] / I \cong A$ be a standard monomial basis with respect to $>$. As before, we take a subset ${\bf u} = \{\u_1, \ldots, \u_r\} \subset C$ and let $M$ be the $r \times n$ matrix matrix whose $j$-th row is $\u_j$, for all $j$. As usual we regard $M$ as a $\Q^r$-weighting matrix on $\k[\x]$, where $\Q^r$ is given the standard lexicographic ordering. 

\begin{proposition}\label{primeconeadapted}
The standard monomial basis $\mathbb{B}$ is adapted to $\v_\w$.  Moreover, we have $\In_>(\In_\w(I)) = \In_>(I)$, that is, $\w \in \textup{GR}^r(I)$ (see Definition \ref{def-GR-r}). 
\end{proposition}
\begin{proof}
By  Lemma \ref{lem-in_w-in_u_i} we have $\In_{\w}(I) = \In_{\u_1}(\ldots \In_{\u_r}(I)\ldots)$.  Since ${\bf u} \subset C \subset \textup{GR}(I)$, by Proposition \ref{midpoint}, we have $\In_\w(I) =  \In_{\u_1+\cdots+\u_r}(I)$. It follows that $\w \in \textup{GR}^r(I)$ and $\In_>(\In_\w(I)) = \In_>(I)$. Proposition \ref{prop-st-monomial-adapted-basis} then implies that $\mathbb{B}$ is adapted to $\v_\w$.
\end{proof}

For $\u \in C$ we can also consider the rank one valuations $\v_\u: A \setminus \{0\} \to \Q$ associated to the weight $\u$. Proposition \ref{prop-st-monomial-adapted-basis} in particular implies that the basis $\mathbb{B}$ is also adapted to $\v_\u$. The valuations $\v_\u$ and the basis $\mathbb{B}$ are related as follows.

\begin{proposition} \label{prop-adapted-basis-multiplicative}
We have the following:
\begin{itemize}
\item[(1)] Let $\u_1, \u_2 \in C$ and $c_1, c_2 \in \Q_{\geq 0}$ and put $\u = c_1\u_1 + c_2\u_2$. Then for any basis element $b \in \mathbb{B}$ we have $\v_\u(b) = c_1 \v_{\u_1}(b) + c_2 \v_{\u_2}(b)$.
\item[(2)] Let ${\bf u} = \{\u_1, \ldots, \u_r\} \subset C$ such that its span has maximal dimension $\dim(C)$. As above let $\v_\w$ denote its associated valuation. Let $b_\alpha$, $b_\beta \in \mathbb{B}$. Consider the expansion of the product $b_\alpha b_\beta$ in the basis $\mathbb{B}$:
$$b_{\alpha}b_{\beta} = \sum c_{\alpha, \beta}^{\gamma} b_{\gamma},$$
where the $c_{\alpha, \beta}^{\gamma} \in \k$. Then for every $b_\gamma$, with $ c_{\alpha, \beta}^{\gamma} \neq 0$, and every $i = 1, \ldots, r$ we have $\v_{\u_i}(b_\gamma) \geq \v_{\u_i}(b_\alpha) + \v_{\u_i}(b_\beta)$.
Moreover, there exits $b_\eta \in \mathbb{B}$, with $c_{\alpha, \beta}^\eta \neq 0$, such that, for every $i$, we have $\v_{\u_i}(b_\eta) = \v_{\u_i}(b_\alpha) + \v_{\u_i}(b_\beta)$. 
\item[(3)] If $\k$ is assumed to be algebraically closed and $C$ is a cone of maximal dimension $d = \dim(A)$, then the 
basis element $b_\eta$ in the part (2) is unique.  
\end{itemize}
\end{proposition}
\begin{proof}
Part (1) and the first assertion in (2) are direct consequences of Proposition \ref{prop-st-monomial-adapted-basis}. 
To prove the second assertion in (2) note that if the equality was not achieved for some $\gamma$, then the product $b_\alpha b_\beta$ would be $0$ in the associated graded $\gr_\w(A)$ which would imply that $\v_\w$ is not a valuation. Finally, by Theorem \ref{th-full-rank-1-dim-leaves}, the assumptions in (3) imply that the valuation $\v_\w$ is a valuation with one-dimensional leaves. This finishes the proof.   
\end{proof}

\section{Prime cones from valuations}
\label{sec-cone-from-subductive-val}
In this section we associate a prime cone to a weight valuation. Any subductive valuation is a weight valuation (\ref{lem-subductiveweight}), so this construction works for all subductive valuations. First, we make an observation which applies to all valuations. 


\begin{proposition}[Higher rank tropicalization map]    \label{prop-values-in-trop-var}
Let $\v: A \setminus \{0\} \to \Q^r$ be a valuation (not necessarily subductive or with a finite Khovanskii basis). Let $\B = \{b_1, \ldots, b_n\} \subset A$ be a set of algebra generators and let $I \subset \k[\x]$ be the ideal of relations among the $b_i$. Let $M = M(\B, \v)$ be the matrix whose columns are $\v(b_1), \ldots, \v(b_n)$. Then $M$ belongs to the rank $r$ tropical variety $\T^r(I)$ (Definition \ref{def-trop-var-higher-rank}). 
\end{proposition}
\begin{proof}
The proof is the same as the usual proof when the valuation has rank $1$. Let $f = \sum_\alpha c_\alpha \x^\alpha \in I$ and suppose by contradiction that $\In_\w(f)$ is the monomial $c_\beta \x^\beta$. Then $\v(f(b_1, \ldots, b_n)) = \v(b_1^{\beta_1} \cdots b_n^{\beta_n}) = M\beta$, which contradicts $f(b_1, \ldots, b_n) = 0$. 
\end{proof}
We think of $\v \mapsto M$ as a higher rank generalization of the tropicalization map in usual tropical geometry. The above (Proposition \ref{prop-values-in-trop-var}) has also been observed in \cite{FosterDhruv}.

Now let us assume that $\Q^r$ is equipped with the standard lexicographic order. 
Let $M \in \Q^{r \times n}$ be a matrix such that the corresponding weight quasivaluation $\v = \v_M: A \setminus \{0\} \to \Q^r$ is indeed a valuation. The next proposition constructs a prime cone $C_\v$ in the tropical variety $\T(I)$ associated to the valuation $\v$.

\begin{proposition}  \label{prop-prime-cone-from-subductive-val}
With notation as above, there exists an open cone $C_\v$ in the tropical variety $\T(I)$ with $\dim(C_\v) \geq \rank(\w)$ such that $\In_\w(I) = \In_\u(I)$, for any $\u \in C_\v$.
Thus, if we denote the common initial ideal $\In_\u(I)$, $\u \in C_\v$, by $\In_{C_\v}(I)$, we have: $$\gr_\v(A) \cong \k[\x] / \In_{C_\v}(I).$$ 
In particular, if $\v$ has maximal rank $d = \dim(A)$ then $C_\v$ has dimension $d$.
\end{proposition}
\begin{proof}
We pass to the homogenization $I_h \subset \k[x_0, \x]$. Let $(0, \w) \in \Q^{m \times (n+1)}$ be the matrix with a $0$ column inserted to the left.  
We regard it as a weighting matrix on $\k[x_0, \x]$.
Proposition \ref{prop-in-u-rep-in-w} implies that there is $\u \in \Q^n$ such that $\In_{(0, \u)}(I_h) = \In_{(0, \w)}(I_h)$, and furthermore  $\In_{\u}(I) = \In_{\w}(I)$. Let $C \in \Sigma(I_h)$ be an open cone that contains $(0, \u)$ and let $G_>(I_h)$ be an appropriate reduced Gr\"obner basis for $I_h$. For any $g \in G_>(I_h)$, the initial form $\In_{(0, \w)}(g)$ is a polynomial $\sum_i p_i(\x)x_0^i$, such that each $p_i(\x)$ is homogeneous with respect to $\w$.  Let $\u_1, \ldots \u_m$ denote the rows of $M$ and let $H$ be its row span over $\Q$. Note that each $p_i(\x)$ is homogeneous with respect to each $\u_j$. It follows that there is some $\epsilon > 0$ such that any $\u'$ in the ball $B_{\epsilon}(0) \subset H$ has the property $\In_{(0,\u) + (0,\u')}(g) = \In_{(0,\u)}(g)$ for all $g \in G_>(I_h)$. Since $\dim(H) = r$ we conclude that $\dim(C \cap \Q^n) \geq r$. The remaining parts of the proposition follow from the fact that if $\v$ is a valuation then $\gr_\v(A)$ is a domain and hence $C_\v$ is a prime cone. This implies that $C_\v \subset \T(I)$ and that it is of dimension at most $d$. 
\end{proof}

\begin{remark}  \label{rem-quasival-cone}
As above let $\B = \{b_1, \ldots, b_n\}$ be a set of algebra generators. Let $\v: A \setminus \{0\} \to \Q^r$ be a valuation and let $M$ be the matrix whose columns are $\v(b_1), \ldots, \v(b_n)$. (1) By \ref{lem-subductiveweight} if we assume that $\v$ is subductive with Khovanskii basis $\B$, then $\v = \v_M$ is a weight valuation. (2) If we do not assume that $\v$ is subductive then by Proposition \ref{prop-values-in-trop-var}, we still can find a cone $C$ in the tropical variety $\T(I)$ such that $\In_\w(I) = \In_C(I)$. But we cannot in general conclude that $\v = \v_M$ and hence we do not know that $\gr_\v(A) \cong \k[\x]/\In_C(I)$. 
\end{remark}

\begin{remark}   \label{rem-st-monomial-basis-adapted-for-v}
The proof of Proposition \ref{prop-prime-cone-from-subductive-val} implies that the cone $C_\v$ can be described using a reduced Gr\"obner basis of the homogenized ideal $I_h$. 
\end{remark}



Let $\v = \v_M: A \setminus \{0\} \to \Q^r$ be a weight valuation with associated prime cone $C_\v$. Let $M'$ be a matrix with rows ${\bf u} = \{\u'_1, \ldots, \u'_s\} \subset C_\v$ such that $rank(M') = \dim(C_\v)$.

\begin{proposition}  \label{prop-S_v-S_v_w-lin-iso}
The value semigroups $S(A, \v_M)$ and $S(A, \v_{M'})$ are linearly isomorphic.  Consequently, the cones $P(A, \v_M)$ and $P(A, \v_{M'})$ are also linearly isomorphic.
\end{proposition}


\begin{proof}
Both associated graded algebras $\gr_{\v_M}(A)$ and $\gr_{\v_{M'}}(A)$ are isomorphic to $\k[\x]/\In_{\u'}(I)$ for any $\u'$ in $C_\v$ by maps which identify the images of the $x_i$. The proof of Proposition \ref{prop-S-lin-iso} implies that the value semigroup of $\v_M$ is isomorphic to the value semigroup of  $\v_{M'}$.
\end{proof}

\section{Compactifications and degenerations} \label{sec-compactifications}
In this section we use an (open) prime cone $C$ in the tropical variety $\T(I)$, for some presentation $A \cong \k[\x] / I$, to construct a compactification of $X = \Spec(A)$. As a byproduct, when the cone $C$ has maximal dimension $d=\dim(A)$, we also get a toric degeneration of $X$.
This construction closely resembles the ``geometric tropicalization'' in \cite{HKT}, \cite{Sturmfels-Tevelev}, \cite{Tevelev}.

We use notation as before. To simplify the discussion we assume that the cone $C$ lies in the negative orthant, i.e. $C \subset \T^-(I) = \T(I) \cap \Q^n_{\leq 0}$. 
We will use results in Section \ref{sec-val-from-prime-cone}. Let ${\bf u} = \{\u_1, \ldots, \u_r\} \subset C$, for simplicity we assume that the $\u_i$ are linearly independent and $r = \dim(C)$. Let $M = M_{\bf u} \in \Q^{r \times n}$ be the matrix whose rows are $\u_1, \ldots, \u_r$. Let $\v_\w$ be the corresponding valuation as constructed in Section \ref{sec-val-from-prime-cone}. We recall that for every $i$, the $i$-th column of $M$ is the vector $\v_M(b_i)$ which in turn is equal to the vector $(\v_{\u_1}(b_i), \ldots, \v_{\u_r}(b_i))$.  Here $\v_{\u_i}$ is the rank $1$ valuation associated to $\u_i \in C$.
We also know that the rank of the valuation $\v_M$ is equal to $\rank(M) = r$ (Propositions \ref{prop-val-from-cone} and \ref{prop-v_w-v_u_i}). Since $C \subset \T^-(I)$ is always contained in the Gr\"obner region $\GR(I)$, by Proposition \ref{primeconeadapted} there is an adapted basis $\mathbb{B}$ for $(A, \v_M)$ (in fact, $\mathbb{B}$ can be taken to be a standard monomial basis for $I$ and some monomial ordering $>$ such that $M \in C_>(I)$).  

After a scaling in $\Q^r$ if needed, we can assume that the value semigroup of $\v_\w$ lies in $\Z^r$.
To construct our compactification, we choose one additional piece of information, namely a lattice point $\delta = (\delta_1, \ldots, \delta_r) \in \Z^r$ which lies in $P^\circ(A, \v_M)$, the relative interior of the Newton-Okounkov cone (Definition \ref{def-Newton-Okounkov-cone}). Given ${\bf u}$ and $\delta$, we construct a projective compactification $\bar{X}_{{\bf u}, \delta} \supset X$. We will give different constructions of this compactification:
\begin{itemize}
\item[(i)] As $\Proj$ of a certain $\Z_{\geq 0}$-graded algebra $T_{{\bf u}, \delta}(A)$.
\item[(ii)] As the GIT quotient, at $\delta$, of an affine variety $E_{\bf u}$ by a natural action of the torus $\Gm^r$. 
\item[(iii)] When the cone $C$ has maximal dimension $r = d = \dim(A)$, we can realize the compactification $\bar{X}_{{\bf u}, \delta}$ as the closure of $X$ embedded into a projective toric variety $Y_{{\bf u}, \delta}$.
\end{itemize}

\subsection{Rees algebras}   \label{subsec-Rees}
In this section we define a generalized Rees algebra $R_{\bf u}(A)$. It plays the main role in the construction of our compactification.

Throughout Section \ref{sec-compactifications}, $\geq$ denotes the partial order on $\Q^r$ defined by comparing the vectors componentwise, namely, $(p_1, \ldots, p_r) \geq (q_1, \ldots, q_r)$ if and only if $p_i \geq q_i$ for all $i$. 

As above, let $\mathbb{B} \subset A$ be the vector space basis for $A$ adapted to the valuation $\v_M$. Beside the valuation $\v_M$ we can consider the rank $1$ valuations $\v_{\u_i}$ corresponding to the vectors 
$\u_i \in {\bf u}$. For ${\bf p} = (p_1, \ldots, p_r)$ we define the set $\mathbb{B}_{\bf p}$ by:
$$\mathbb{B}_{\bf p} = \{b \in \mathbb{B} \mid \v_{\u_i}(b) = p_i,~ \forall i=1, \ldots, r\} \subset \mathbb{B}.$$
We note that since $C \subset \Q^n_{\leq 0}$, the set $\mathbb{B}_{\bf p}$ is nonempty only for ${\bf p} \in \Z^r_{\leq 0}$. For ${\bf p} \in \Z_{\leq 0}^r$ we then define the subspaces $W_{\bf u}({\bf p})$ and $F_{\bf u}({\bf p}) \subset A$ as follows: $$W_{\bf u}({\bf p}) = \span(\mathbb{B}_{\bf p}),$$
$$ F_{\bf u}({\bf p}) = \bigoplus_{{\bf q} \geq {\bf p}} W_{\bf u}({\bf q}) = \span(\bigcup_{{\bf q} \geq {\bf p}} \mathbb{B}_{\bf q}).$$
Clearly, $\{F_{\bf u}({\bf p})\}_{{\bf p} \in \Z_{\leq 0}^r}$ is a multiplicative filtration of $A$ (see Proposition \ref{prop-adapted-basis-multiplicative}). 

Consider the Laurent polynomial algebra $A[{\bf t}^\pm] = A[t_1^\pm, \ldots, t_r^\pm]$, where we have used ${\bf t}$ as an abbreviation for the indeterminates $(t_1, \ldots, t_r)$. Also, given ${\bf p} = (p_1, \ldots, p_r)$ we will write 
${\bf t}^{\bf p}$ to denote the monomial $t_1^{p_1} \cdots t_r^{p_r}$. The Rees algebra $R_{\bf u}(A)$ is the subalgebra of the Laurent polynomial algebra 
$A[{\bf t}^\pm]$ defined by:
$$R_{\bf u}(A) = \bigoplus_{{\bf p} \in \Z_{\leq 0}^r} F_{\bf u}({\bf p}) {\bf t}^{-{\bf p}} \subset \bigoplus_{{\bf p} \in \Z^r} A{\bf t}^{-{\bf p}} = A[{\bf t}^\pm].$$

Let $\phi: \k[{\bf t}] \to R_{\bf u}(A)$ be the homomorphism obtained by sending $t_i$ to $1t_i \in F_{\bf u}(-e_i)t_i$, for all $i$. Here $e_i = (0, \ldots, 1, \ldots, 0)$ denotes the $i$-th standard basis element in $\Z^r$. The homomorphism $\phi$ gives $R_{\bf u}(A)$ the structure of a $\k[{\bf t}]$-module.
We let $E_{\bf u}$ denote the affine scheme $\Spec(R_{\bf u}(A))$ defined by this Rees algebra. The next proposition establishes basic properties of the scheme $E_{\bf u}$ and its relationship to $X$. We leave the proof of this proposition to the reader (it is very similar to \cite[Proposition 2.2]{Teissier}).  

\begin{proposition}   \label{prop-Rees}
As above, let $R_{\bf u}(A)$ be the Rees algebra of $A$ with respect to ${\bf u} \subset C$. We then have:
\begin{itemize}
\item[(1)] The map $\phi: \k[{\bf t}] \to R_{\bf u}(A)$ defines a flat family $\pi: E_{\bf u} \to \mathbb{A}^r$.  
\item[(2)] There is a natural action of the torus $\Gm^r$ on $E_{\bf u}$ which lifts the natural action of $\Gm^r$ on $\mathbb{A}^r$,
\item[(3)] $\Spec(A[{\bf t}^\pm]) = X \times \Gm^r$ is the complement of the hypersurface $V_{\bf u} \subset E_{\bf u}$ defined by the equation $t_1\cdots t_r = 0$.
\end{itemize}
\end{proposition}

Next, we describe the fibers of the map $\pi: E_{\bf u} \to \mathbb{A}^r$.  
By Proposition \ref{prop-Rees}(2, 3), for any $c=(c_1, \ldots, c_r)$ with $c_i \neq 0$, for all $i$, the fiber $\pi^{-1}(c)$ is isomorphic to $X$. We will see below that all other fibers of the family are all degenerations of $X$ coming from subsets of ${\bf u}$. Let $\sigma \subset {\bf u}$. Analogous to $W_{\bf u}({\bf p})$ and $F_{\bf u}({\bf p})$, given ${\bf p}'=(p'_i)_{\u_i \in \sigma} \in \Z_{\leq 0}^{|\sigma|}$ we can define $W_{\sigma}({\bf p}') = \span\{b \in \mathbb{B} \mid \v_{\u_i}(b) = p'_i,~\forall \u_i \in \sigma\}$ and $F_{\sigma}({\bf p}') = \span\{b \in \mathbb{B} \mid \v_{\u_i}(b) = p'_i,~\forall \u_i \in \sigma\}$. We then have the direct sum decomposition $A = \bigoplus_{{\bf p}'} W_{\sigma}({\bf p}')$. Let $\gr_{\sigma}(A)$ denote the associated graded algebra of the filtration 
$\{F_\sigma({\bf p}')\}_{{\bf p}' \in \Z_{\leq 0}^{|\sigma|}}$ on $A$. Proposition \ref{prop-adapted-basis-multiplicative} implies that $\gr_{\sigma}(A)$ is a domain and hence this filtration must come from a valuation (see Section \ref{subsec-quasival-filtration}).

\begin{proposition} \label{prop-Rees-fibers}
Let $c_\sigma = (c_{\sigma, 1}, \ldots, c_{\sigma, r}) \in \Z^r$ be the vector of $0$'s and $1$'s defined by: 
$$c_{\sigma, i} = \begin{cases} 0 & \u_i \in \sigma \\ 1 &
\u_i \notin \sigma, \end{cases}$$
for all $i = 1, \ldots, r$.
Then the fiber $\pi^{-1}(c_{\sigma})$ is isomorphic to $X_{\sigma} = \Spec(\gr_{\sigma}(A))$.  
\end{proposition}
\begin{proof}
Specializing $t_i =1$ for $\u_i \notin \sigma$ yields an algebra graded by $\Z_{\leq 0}^{|\sigma|}$ . For every ${\bf p}' \in \Z_{\leq 0}^{|\sigma|}$, the corresponding graded component is:
$$F_{\sigma}({\bf p}') = \sum_{{\bf p}|_{\sigma} = {\bf p}'} F_{\bf u}({\bf p}).$$
Also, specializing $t_i = 0$ for $u_i \in \sigma$ yields an algebra graded by $\Z_{\leq 0}^{|\sigma|}$ with the graded components:
\begin{equation}
F_{\sigma}({\bf p}')/ \sum_{{\bf q}' \geq {\bf p}'} F_{\sigma}({\bf q}'),
\end{equation}
for ${\bf p}' \in \Z_{\leq 0}^{|\sigma|}$. It is straightforward to check that $F_\sigma({\bf p}')$ has a vector space basis consisting of the images of those $b \in \mathbb{B}$ with $\v_{u_i}(b) = p_i$ for $u_i \in \sigma$. It follows that these graded components can be identified with the space $W_{\sigma}({\bf p}')$.  We leave it to the reader to check that the multiplication operation is likewise the same as in $\gr_\sigma(A)$.
\end{proof}

Proposition \ref{prop-Rees-fibers} and Proposition \ref{prop-Rees}(2) imply the following: 
\begin{corollary}[Toric degeneration of $X$]   \label{cor-toric-degen}  
All fibers of the family $\pi: E_{\bf u} \to \mathbb{A}^r$ are reduced and irreducible. Moreover, these fibers are degenerations of $X$ corresponding to valuations constructed from subsets of ${\bf u}$. In particular, the fiber over the origin is $\Spec(\gr_{\v_M}(A))$. Moreover, 
if $r = d = \dim(A)$, the fiber over the origin is $\Spec(\k[S(A, \v_{\w})]$ which is a (not necessarily normal) affine toric variety (see Proposition \ref{prop-grA-semigp-algebra}). 
\end{corollary}

\begin{remark}  \label{rem-comp-totalorder-vs-partialorder}
(1) In all the above constructions/definitions, instead of the partial order $>$ we can use a group ordering $\succ$ on $\Z^r$ which refines $>$. For example we can take $\succ$ to be a lexicographic order. It follows from Proposition \ref{prop-adapted-basis-multiplicative} that the resulting associated graded algebra $\gr_\succ(A)$ is the same as the associated graded algebra $\gr_{\bf u}(A)$ corresponding to the filtration by the $F_{\bf u}({\bf p})$.
Note that the associated graded $\gr_\succ(A)$ is in fact the associated graded $\gr_{\v_\w}(A)$ of the valuation $\v_\w$. 

(2) As far as the authors know, the construction of the Rees algebra associated to a valuation is due to B. Teissier (see \cite{Teissier} and in particular Proposition 2.2 in there which is very close to our Proposition \ref{prop-Rees}). Also \cite[Proposition 3]{Anderson} is a $1$-parameter version of \cite[Proposition 2.2]{Teissier}. 
 \end{remark}

\subsection{The compactification of $X$}\label{defcompact}
We can now construct the compactification $\bar{X}_{{\bf \u}, \delta}$ of $X$.
Since the Rees algebra $R_{\bf u}(A)$ is by definition $\Z^r$-graded, the scheme $E_{\bf u} = \Spec(R_{\bf u}(A))$ comes with a natural action of the torus $\Gm^r$. We define  $\bar{X}_{{\bf u}, \delta}$ to be the GIT quotient $E_{\bf u} \q_{\delta} \Gm^r$. Equivalently, from definition of the GIT quotient, we can realize this scheme as $\Proj$ of the $\Z_{\geq 0}$-graded subalgebra $$T_{{\bf u}, \delta}(A)= \bigoplus_{N \geq 0} F_{\bf u}(N\delta) {\bf t}^{-N\delta} \subset R_{\bf u}(A).$$  The affine scheme $E_{\bf u}$ is of finite type, so it follows that $\bar{X}_{{\bf u}, \delta}$ is projective. 

We define valuation $\tilde{\v}_\w: T_{{\bf u}, \delta}(A) \to \Z_{\geq 0} 
\times \Z_{\leq 0}^r$ as follows. For any $0 \neq f = \sum_{i=0}^m f_i{\bf t}^{-i\delta}$, we let: $$\tilde{\v}_\w(f) = (N, \v_\w(f_N)),$$ where $N = \MIN\{ i \mid f_i \neq 0\}$. Similarly, for any $\u \in C$ we can define valuation $\tilde{\v}_\u: T_{{\bf u}, \delta}(A) \to \Z_{\geq 0} \times \Q_{\leq 0}$. We observe that the value semigroup of $\tilde{\v}_\w$ is:
\begin{equation} \label{equ-S-tilde-v}
S(T_{{\bf u}, \delta}(A), \tilde{\v}_\w) = \{(N, {\bf p}) \mid {\bf p} \in S(A, \v_\w),~ {\bf p} \geq N\delta\}.
\end{equation}
Since $S(A, \v_\w)$ is finitely generated as a semigroup, generated by the columns of $M$, it follows that $S(T_{{\bf u}, \delta}(A))$ is also finitely generated.  

The algebra $T_{{\bf u}, \delta}(A)$ has a natural vector space basis $$\tilde{\mathbb{B}} = \{ b{\bf t}^{-N\delta} \mid b \in \mathbb{B}_{\bf q},\, {\bf q} \geq N\delta \}.$$ This basis is adapted to the valuations $\tilde{\v}_\w$ and 
$\tilde{\v}_\u$, for all $\u \in C$. 

Note that by definition the valuation $\tilde{\v}_\w$ is homogeneous with respect to the $\Z_{\geq 0}$-grading on $T_{{\bf u}, \delta}(A)$. Hence we can  consider the Newton-Okounkov body $\Delta_{{\bf u}, \delta} = \Delta(T_{{\bf u}, \delta}(A), \tilde{\v}_\w)$. By \eqref{equ-S-tilde-v} we have:
\begin{equation}  \label{equ-Delta-u}
\Delta_{{\bf u}, \delta} = \{ {\bf p} \in \R_{\leq 0}^r \mid {\bf p} \geq \delta \} \cap P(A, \v_\w),
\end{equation}
where $P(A, \v_\w)$ is the Newton-Okounkov cone of $(A, \v_\w)$, that is, the cone generated by the columns of the matrix $M$.

\begin{remark}(Toric degeneration of $\bar{X}_{{\bf u}, \delta}$)  
\label{rem-toric-degen-bar-X}
When $r = d$, the associated graded algebra of the valuation $\tilde{\v}_\w$ is isomorphic to the semigroup algebra $\k[S(T_{{\bf u}, \delta}(A), \tilde{\v}_\w)]$. Moreover, for any $\u \in C$, the associated graded algebra of $\tilde{\v}_\u$ is also isomorphic to this semigroup algebra. It follows that we have a toric degeneration of $\bar{X}_{{\bf u}, \delta}$ to the projective toric variety $\Proj(\k[S(T_{{\bf u}, \delta}(A), \tilde{\v}_\w)])$. The normalization of this toric variety is the toric variety associated to the polytope $\Delta_{{\bf u}, \delta}$.
\end{remark}

For the remainder of this section we assume that $r=d$. In this case, we construct an embedding of $X$ into a projective toric variety $Y_{{\bf u}, \delta}$ such that $\bar{X}_{{\bf u}, \delta}$ is the closure of $X$ in $Y_{{\bf u}, \delta}$. 

Let us define the semigroup:
$$\hat{S}_\delta = \{ (N, {\bf a}) \mid N \in \Z_{\geq 0}, {\bf a} \in \Z^n, M{\bf a} \geq N\delta \} \subset \Z_{\geq 0} \times \Z^n.$$
Recall that if ${\bf a} = (a_1, \ldots, a_n)$ then $M{\bf a} = \sum_i a_i\v_{\w}(b_i)$ and as before $\geq$ denotes the partial order on $\Z^r$ given by componentwise comparison of vectors. Note that from definition, $\hat{S}_\delta$ is a saturated semigroup. 

Let $Y_{{\bf u}, \delta} = \Proj(\k[\hat{S}_\delta])$ with respect to the grading by $N \in \Z_{\geq 0}$. It is the projective toric variety associated to the polytope 
$$\hat{\Delta}_{{\bf u}, \delta} = \{{\bf a} \mid M{\bf a} \geq \delta \} \subset \R_{\leq 0}^n.$$
In other words, $\hat{\Delta}_{{\bf u}, \delta}$ is the polytope defined by 
the inequalities ${\bf a} \leq 0$ and ${\bf a} \cdot u_i \geq \delta_i$ for all $i=1, \ldots, r$. Note that $0$ is a vertex of the polytope $\hat{\Delta}_{{\bf u}, \delta}$ and the cone at this vertex is the negative orthant. Thus we can consider the toric variety $Y_{{\bf u}, \delta}$ as a compactification of the affine space $\mathbb{A}^n$.

There is a natural homomorphism $\hat{\pi}: \k[\hat{S}_\delta] \to T_{{\bf u}, \delta}(A)$ which sends $(N, {\bf a})$ to $(\prod_i b_i^{a_i}) {\bf t}^{-N\delta}$, where ${\bf a} = (a_1, \ldots, a_n)$. One verifies that $\hat{\pi}$ is indeed surjective. 
We have the following proposition. We omit the straightforward proof. 
\begin{proposition}   \label{prop-toric-compact}
The scheme $Y_{{\bf u}, \delta}$ is the normal projective toric variety associated to $\hat{\Delta}_{{\bf u}, \delta}$, and is a compactification of $\mathbb{A}^n$.  Furthermore, the closure of $X \subset \mathbb{A}^n \subset Y_{{\bf u}, \delta}$ is $\bar{X}_{{\bf u}, \delta}$.
\end{proposition}

\subsection{The divisor at infinity $D_{{\bf u}, \delta}$}
{We let $D_{{\bf u}, \delta}$ be the divisor in $\bar{X}_{{\bf u}, \delta}$ defined by the ideal $I_{\bf u} = \langle {\bf t}^{-\delta} \rangle \cap T_{{\bf u}, \delta}(A)$.
(Here the ideal generated by ${\bf t}^{-\delta}$ means with respect to the $\k[{\bf t}]$-module structure on the Rees algebra $R_{{\bf u}, \delta}(A)$ given by map $\phi$ in Section \ref{subsec-Rees} after the definition of $R_{{\bf u}, \delta}(A)$.)}

To simplify the discussion, we assume that $\delta = (-1, \ldots, -1)$. Let us see that this is always possible. Since $\delta$ is in the relative interior of $P(A, \v_\w)$ we can find $v \in \Q_{\geq 0}^r$ such that $Mv = -\eta$.  Now let $u_i' = \frac{\eta_1\cdots \eta_d}{\eta_i}u_i$ and $v' = \frac{v}{\eta_1\cdots \eta_d}$. Now if $M'$ is the matrix whose rows are the $u_i'$, it easy to check that $M'v' = (-1, \ldots, -1)$.  

Under the assumption $\delta = (-1, \ldots, -1)$, we have that $D_{{\bf u}, \delta}$ is the divisor at infinity $\bar{X}_{{\bf u}, \delta} \setminus X$. The purpose of this section is to show that:
\begin{itemize}
\item[(1)] The divisor $D_{{\bf u}, \delta}$ has combinatorial normal crossings.
\item[(2)] The vectors $\u_i \in {\bf u}$ are in one-to-one correspondence with the irreducible components of  $D_{{\bf u}, \delta}$, and moreover, for every $\u_i \in {\bf u}$, the valuation $\v_{\u_i}$ is given by the order of vanishing along the corresponding irreducible component of $D_{{\bf u}, \delta}$.
\end{itemize}

Before we proceed with the proofs, we need few more definitions. Let $\sigma \subset {\bf u} = \{u_1, \ldots, u_r\}$. To $\sigma$ we associate the ideal $I_\sigma \subset T_{{\bf u}, \delta}(A)$ defined by:
$$I_\sigma = \langle t_i \mid u_i \in \sigma \rangle \cap T_{{\bf u}, \delta}(A).$$ We note that since for every $i$, $\langle t_i \mid u_i \in {\bf u} \rangle \subset A[{\bf t}]$ is a prime ideal, the ideal $I_\sigma \subset T_{{\bf u}, \delta}(A)$ is also prime. In particular, for every $i$, we let $I_i = \langle t_i \rangle \cap T_{{\bf u}, \delta}(A)$ and we denote by $D_i$ the divisor in $\bar{X}_{{\bf u}, \delta}$ defined by the ideal $I_i$.

\begin{proposition}   \label{prop-div-at-infinity}
$D_{{\bf u}, \delta} = \sum_i D_i$ is a divisor in $\bar{X}_{{\bf u}, \delta}$ with combinatorial normal crossings.
\end{proposition}
\begin{proof}
It is straightforward to check that the ideals $I_\sigma = \langle t_i \mid u_i \in \sigma \rangle \cap T_{{\bf u}, \delta}(A)$ are distinct and prime. Thus it suffices to show that the subscheme in $\bar{X}_{{\bf u}, \delta}$ defined by the largest ideal $I_{\bf u}$ has the correct codimension equal to $r$. The subscheme in $E_{\bf u}= \Spec(R_{\bf u}(A))$ defined by the ideal $\langle t_i \mid i=1, \ldots, r \rangle$ is the fiber $\pi^{-1}(0)$ over $0$ of the family $\pi: E_{\bf u} \to \mathbb{A}^r$. By Proposition \ref{prop-Rees-fibers}, the fiber $\pi^{-1}(0)$ is isomorphic to $\Spec(\gr_{\v_M}(A)$ of the valuation $\v_M$. We also know that the cone generated by the $\Gm^r$-weights of the algebra $\gr_{\v_M}(A)$ is the Newton-Okounkov cone $P(A, \v_M)$. This cone has dimension $r = \rank(M)$ which is equal to dimension of the prime cone $C$ we started with. Now since the fiber $\pi^{-1}(0)$ is stable under the $\Gm^r$-action on $E_u$, we conclude that the subscheme in $\bar{X}_{{\bf u}, \delta}$ defined by $I_{\bf u}$ is isomorphic to the GIT quotient $\pi^{-1}(0) //_\delta \, \Gm^r$. The claim now follows from the following lemma from geometric invariant theory. 
\begin{lemma}  \label{lem-GIT}
Let $R$ be a finitely generated algebra with Krull dimension $d$ and equipped with a rational $\Gm^r$-action where $r \leq d$. Suppose the cone $C(R) \subset \Q^r$ generated by the weights of the $\Gm^r$-action has maximal dimension $r$. Let $\delta$ be a weight which lies in the interior of $C(R)$. Then the GIT quotient $\Spec(R) //_\delta\, \Gm^r$ has dimension $d - r$.  
\end{lemma}
\end{proof}

\begin{corollary}   \label{cor-div-at-infinity-ord-vanishing}
For every $\u_i \in {\bf u}$ the valuation $\v_{\u_i}$ is given by the order of vanishing along the divisor $D_i$. That is, for any $0 \neq f \in A$ the value 
$\v_{\u_i}(f)$ is equal to order of zero/pole of $f$ along the divisor $D_i$.
\end{corollary}
\begin{proof}
Let $\chi$ be a character of $\Gm^r$ of weight $-\delta$, so that $T_{{\bf u}, \delta} = [R_{\bf u}(A) \otimes \k[\chi]]^{\Gm^r}$.  The the ideal $\langle t_i \rangle \subset R_{\bf u}(A)\otimes k[\chi]$ is $\Gm^r$-stable and the unique maximal ideal in the local ring at this ideal can be generated by the invariant $t_1\cdots t_r\chi$ for all $i$.  As a consequence it generates the maximal ideal in the local ring at $I_i \subset T_{{\bf u}, \delta}(A)$, and it follows that the $D_i$ degree of any regular function $f \in A$ on $X$ can be computed by taking the $t_i$-degree, as $f = \frac{\bar{f}}{\prod_j t_j^{\v_{\u_j}(f)}} \in A \subset \frac{1}{t_1\cdots t_r}R_{\bf u}(A)$ for $\bar{f} = f\prod_j t_j^{\v_{\u_j}(f)} \in R_{\bf u}(A)$. We obtain that this degree is $\v_{u_i}(f)$.
\end{proof}

\section{Examples} \label{sec-examples}
\begin{example}[The wonderful compactification of an adjoint group $G$]   \label{ex-wonderful-comp}
Let $G$ be an adjoint form of a semisimple algebraic group over an algebraically closed characteristic $0$ field $\k$. We show that the wonderful compactification $\overline{G}$ can be realized by means of the compactification construction outlined in Section \ref{defcompact}. 

We pick a system of simple roots $\alpha_1, \ldots, \alpha_r$, these generate the root lattice $\mathcal{R}$.  Let $h_1, \ldots, h_r \in \mathfrak{h}$ be the corresponding coroots, recall that these pair with the weights $\Lambda$ such that $\omega_i(h_j) = \delta_{ij}$ for the fundamental weights $\omega_i$.   The $\omega_1, \ldots, \omega_r$ generate the monoid of dominant weights $\Lambda_+ \subset \Lambda$.  Finally, let $\hat{\alpha}_1, \ldots, \hat{\alpha}_r$ be the fundamental coweights; these have the property that $\alpha_j(\hat{\alpha}_i) = \delta_{ij}$.  In particular if $\omega \prec \eta$ in the dominant weight ordering, we must have $(\eta - \omega)(\hat{\alpha}_i) \geq 0$ for each $1 \leq i \leq r$. 

 The coordinate ring $\k[G]$ is known to have the following direct sum decomposition:

\begin{equation}
\k[G] = \bigoplus_{\lambda \in \Lambda_+} \End(V(\lambda)).
\end{equation}
Here $V(\lambda)$ is the irreducible representation of $G$ associated to  $\lambda \in \Lambda_+$.  Each fundamental coweight defines a rank $1$ $(G\times G)$-invariant valuation $v_i: \k[G]\setminus \{0\} \to \Z$, where the filtration defined by $v_i$ is by the subspaces $F^i_{\leq m} = \bigoplus_{\eta(\hat{\alpha}_i) \leq m} \End(V(\eta)) \subset \k[G]$. Let $R(G)$ denote the Rees algebra associated to the $v_i$:
\begin{equation}
R(G) = \bigoplus_{{\bf p} \in \Z^r_{\leq 0}} \bigoplus_{\lambda(\hat{\alpha}_i) \leq p_i}  \End(V(\lambda)){\bf t}^{-{\bf p}}.
\end{equation}

\begin{remark}
This Rees algebra has been considered by Popov \cite{Popov} in the context of the \emph{horospherical contraction} of a $G$ variety. 
\end{remark}

We select a weight $\rho$ with the property that $\rho(h_i) > 0$, and we let $\delta = (\rho(\hat{\alpha}_1), \ldots, \rho(\hat{\alpha}_r))$.  Finally, we define $T_{\delta}(G)$ as in Section \ref{defcompact}:
\begin{equation}
T_{\delta}(G) = \bigoplus_{N} \bigoplus_{\lambda(\hat{\alpha}_i) \leq N\delta_i} \End(V(\lambda))t^N.
\end{equation}

  There is an ample $(G\times G)$-line bundle $\mathcal{L}_{\rho}$ on the wonderful compactification $\overline{G}$ corresponding to $\rho$.  Global sections of $\mathcal{L}_{\rho}$ have the following description:

\begin{equation}
H^0(\overline{G}, \mathcal{L}_{\rho}^{\otimes N}) = \bigoplus_{\lambda \prec N\rho} \End(V(\lambda)).\\
\end{equation} 

We claim that $H^0(\bar{G}, \mathcal{L}_{\rho}^{\otimes N}) = \bigoplus_{\lambda(\hat{\alpha}_i) \leq N\delta_i} \End(V(\lambda))t^N.$  For $\lambda \prec N\rho$ we must have $N\rho - \lambda = \sum n_i\alpha_i$ for $n_i \in \Z_{\geq 0}$; it follows that $H^0(\bar{G}, \mathcal{L}_{\rho}^{\otimes N}) \subset \bigoplus_{\lambda(\hat{\alpha}_i) \leq N\delta_i} \End(V(\lambda))t^N$.  For the other inclusion, suppose more generally that $(\eta - \lambda)(\hat{\alpha}_i) \geq 0$ for all $i$.  The group $G$ is adjoint and hence $\mathcal{R} = \Lambda$, so it follows that $\eta - \lambda = \sum n_i \alpha_i$ for $n_i \in \Z$.  If any of these coefficients were negative, the corresponding $\hat{\alpha}_i$ would likewise evaluate to a negative integer.  Consequently, we must have $\overline{G} = \Proj(T_{\delta}(G))$. 
\end{example}

\begin{example}\label{GZpatternexample}[Gel'fand-Zetlin patterns and the Pl\"ucker algebra]
Let $\k \subset F$ be a transcendental field extension, and $\v: F\setminus \{0\} \to \Z^d$ a valuation of rank equal to the transcendence degree of $F$ over $\k$.  It is natural to ask when a finite subset $\mathcal{B} \subset F$ is a Khovanskii basis for the $\k$-algebra $\k[\mathcal{B}] \subset F$ generated by $\mathcal{B}$, with respect to $\v$.   Let $\mathcal{X}$ be an $n \times n$ array of indeterminates $x_{ij}$, and let $\k(\mathcal{X})$ be the quotient field of the polynomial algebra $\k[\mathcal{X}]$.   A rank $n^2$ valuation $\v$ can be defined on $\k(\mathcal{X})$ by ordering the monomials in $\k[\mathcal{X}]$ lexicographically using the row-wise ordering of the entries of $\mathcal{X}$:

\begin{equation}
x_{11} > \ldots > x_{1n} > x_{21} > \ldots > x_{2n} > \ldots > x_{nn}.\\
\end{equation}

Let $\sigma \subset \{1, \ldots, n\}$ be an ordered subset, and let $p_{\sigma} \in \k(\mathcal{X})$ be the form obtained by taking the determinant of the $|\sigma| \times |\sigma|$ minor of $\mathcal{X}_{\sigma}$ composed of the $x_{ij}$ with $1 \leq i \leq |\sigma|$ and $j \in \sigma$. Let $\mathcal{P}$ denote the set of all the $p_\sigma$. The algebra $\k[\mathcal{P}] \subset \k(\mathcal{X})$ is known as the Pl\"ucker algebra.  It is the coordinate ring of the quotient variety $\GL_n(\k)/U$, where $U \subset \GL_n(\k)$ is the unipotent group of upper triangular matrices with $1$'s on the diagonal, 
and also it can be identified with the total coordinate ring of the full flag variety 
$\mathcal{F\ell}_n$.   

The valuation $\v$ induces a maximal rank valuation on $\k[\mathcal{P}]$ with Khovanskii basis $\mathcal{P}$, see \cite[Theorem 14.11]{Miller-Sturmfels}.  It is also known 
(see e.g. \cite[Theorem 14.23]{Miller-Sturmfels}) that the associated graded algebra $\gr_{\v}(\k[\mathcal{P}])$ is isomorphic to affine semigroup algebra of the Gel'fand-Zetlin pattern semigroup $\textup{GZ}_n \subset \Z_{\geq 0}^{\binom{n+1}{2}}$.  An element $w \in \textup{GZ}_n$ is a triangular array of $\binom{n+1}{2}$ non-negative integers $w_{i,j}$ $1 \leq i \leq n$, $1 \leq j \leq n+1 - i$, organized into $n$ rows of lengths $n, n-1, \ldots, 1$.  These entries satisfy the inequalities $w_{i,j} \geq w_{i-1,j} \geq w_{i,j-1}$. 

The ideal $I$ of polynomial relations among the $p_{\sigma}$ is generated by a set of quadratic polynomials $G$ which is best described combinatorially; we follow the presentation in \cite[Theorem 14.6]{Miller-Sturmfels}.   There is a partial order on the $\sigma \subset \{1, \ldots, n\}$, where $\sigma \prec \tau$ if $|\sigma| \leq |\tau|$ and 
$\sigma_i \leq \tau_i$ for all $1 \leq i \leq |\sigma|$ (here $\sigma_i$ and $\tau_i$ denote the $i$-th elements in increasing order in $\sigma$ and $\tau$ respectively). If $\sigma$ and $\tau$ are incomparable with respect to $\succ$ and $s = |\sigma| \geq |\tau| = t$ there is some index $j$ with $\sigma_j > \tau_j$.  Let $g_{\sigma, \tau}$ be the following polynomial, where the sum is taken over all permutations $\mathcal{S}_{\sigma,\tau}$ of the $s+1 $ indices $\tau_1, \ldots, \tau_j, \sigma_j, \ldots, \sigma_s$:

\begin{equation}
g_{\sigma, \tau} = \sum_{\pi \in \mathcal{S}_{\sigma,\tau}} \textup{sign}(\pi)p_{\pi(\sigma)}p_{\pi(\tau)}.
\end{equation}

The valuation $\v$ induces a partial ordering on the monomials in the variables $\mathcal{P}$, this partial ordering can be completed to a monomial ordering by first ordering with $\v$ and then ordering with the reverse lexicographic ordering induced by the total ordering on the $\sigma$ where $\sigma < \tau$ if $|\sigma| > |\tau|$ or $|\sigma| = |\tau|$ and $\sigma$ comes before $\tau$ in the lexicographic ordering on subsets of $\{1, \ldots, n\}$. We call this concatenated ordering $>_{\v}$.  We have made this choice so that \cite[Theorem 14.6]{Miller-Sturmfels} implies that $G$ is a Gr\"obner basis with respect to $>_{\v}$ and $\In_{>}(\In_{\v}(g_{\sigma\tau})) = \In_{>}(g_{\sigma\tau}) = p_{\sigma}p_{\tau}$ for each $g_{\sigma\tau} \in G$. 

Now, following \cite[Theorem 14.16]{Miller-Sturmfels}, we consider the reduced Gr\"obner basis $G^{red}$ associated to the $G$. 
The partial order $\succ$ defines a lattice on the $\sigma$, let $\wedge$ and $\vee$ be the meet and join in this lattice. The initial ideal $\In_{\v}(I)$ is generated by the binomial initial forms of the members of $G^{red}$:

\begin{equation}
p_{\sigma}p_{\tau} - p_{\sigma\wedge \tau}p_{\sigma \vee \tau},\\
\end{equation}

\noindent 
Here $\sigma$ and $\tau$ are incomparable under $\succ$.  Proposition \ref{prop-prime-cone-from-subductive-val} then implies the following. 

\begin{proposition}
Let $C_{\textup{GZ}}$ be the cone of weights $\u \in \Q^{2^{n-1}}$ in the Gr\"obner fan of $I$ defined by the inequalities implied by the following:

\begin{equation}
\In_{\u}(h_{\sigma\tau}) = p_{\sigma}p_{\tau} - p_{\sigma\wedge \tau}p_{\sigma \vee \tau}, \ \ h_{\sigma\tau} \in G^{red},\\
\end{equation}

\noindent

For any $\u \in C_{\textup{GZ}}^\circ$, $\In_{\u}(I) = \In_{\v}(I)$, so $C_{\textup{GZ}}$ is a prime cone in the tropical variety of $I$.
\end{proposition}

Finding explicit inequalities describing the prime cone $C_{\textup{GZ}}$ is more involved, an answer to this problem can be found in \cite[Theorem 6.2]{Makhlin}.

\end{example}

By Proposition \ref{prop-S_v-S_v_w-lin-iso}, the $\Z_{\geq 0}$ column span of any matrix $M_{\textup{GZ}}$ with rank equal to the dimension of $\k[\mathcal{P}]$ and rows taken from $C^\circ_{\textup{GZ}}$ is isomorphic to the Gel'fand-Zetlin patterns as a semigroup.  

\begin{example}[$\Gr_3(\C^6)$]
We show that Algorithm \ref{algo-find-Khovanskii-basis} can find a Khovanskii basis of the $\Gr_3(\C^6)$ Pl\"ucker algebra for a valuation defined in \cite{M-NOK} using the  Pl\"ucker generators as input.

In \cite{M-NOK}, the second author defines a family of maximal rank valuations on the coordinate ring $\C[P_n(\SL_3(\C))]$ of the configuration space $P_n(\SL_3(\C)) = \SL_3(\C) \ql [\SL_3(\C)/U]^n$.  Here the right quotients are by $U \subset \SL_3(\C)$, a maximal unipotent subgroup, and the left quotient is the geometric invariant theory quotient by the diagonal action of $\SL_3(\C)$. The projective coordinate ring $\C[\Gr_3(\C^n)]$ of the Grassmannian variety of $3$-planes in $\C^n$ with respect to its Pl\"ucker embedding is naturally realized as a subalgebra of $\C[P_n(\SL_3(\C))]$, and therefore inherits these maximal rank valuations. 

For a particular selection of combinatorial parameters, the value semigroup of one of these valuations $\v: \C[\Gr_3(\C^6)] \to \Z^{36}$ is a sub-semigroup $\textup{BZ}_{\T}(3, 6)$ of the Berenstein-Zelevinsky quilts for a particular choice of trivalent $6$-leaf tree $\T$.  This semigroup is studied in \cite{MZ}, and examples of members of this semigroup are depicted in Figure \ref{quiltexample}.  

 By \cite{MZ}, $\textup{BZ}_{\T}(3, 6)$ is generated by the images $\v(p_{ijk})$ of the $\binom{6}{3}$ Pl\"ucker coordinate functions, and the image $\v(T)$ of one additional coordinate function $T$.  Furthermore, it can be shown that the image of $T$ (in the associated graded) is not expressible as a monomial in the images of the Pl\"ucker generators. However, $T$ itself can be written as a binomial in the Pl\"ucker generators:
 
\begin{equation}\label{binomialinvariant}
T = p_{135}p_{246}- p_{235}p_{146}.
\end{equation}

It can be shown that the following equation holds in the value semigroup $\textup{BZ}_{\T}(3, 6)$:
\begin{equation}
\v(p_{135}) + \v(p_{246}) = \v(p_{235}) + \v(p_{146}).
\end{equation}

It follows that $T$, and therefore a Khovanskii basis for $\C[\Gr_3(\C^6)]$ with respect to $\v$, can be found with one application of Algorithm \ref{algo-find-Khovanskii-basis} acting on the Pl\"ucker generators. 

\begin{figure}[htbp]
\centering
\includegraphics[scale = 0.5]{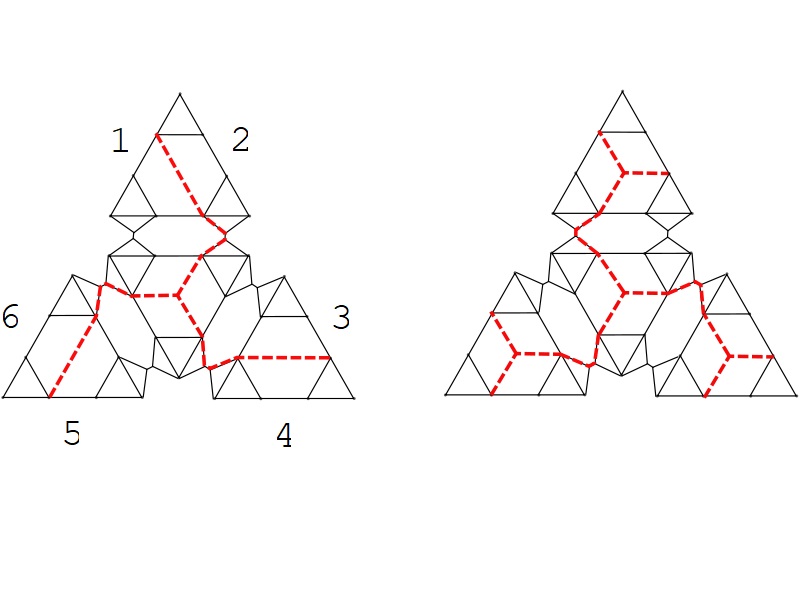}
\caption{Left: $\v(p_{135})$. Right: $\v(T)$.}
\label{quiltexample}
\end{figure}
\end{example}

\begin{example}[The trace algebra on two letters]
The set of representations $\mathcal{X}(F_2, \SL_2(\C))$ of the rank $2$ free group $F_2$ into the complex group $\SL_2(\C)$ can be given the structure of an irreducible complex variety. This variety can be constructed as the image of $\SL_2(\C)\times \SL_2(\C)$ under the polynomial map which sends a pair $(M, N) \in \SL_2(\C)\times \SL_2(\C)$ to the $3$-tuple of traces $t_1 = \tr(M), t_2 = \tr(N), x_1 = \tr(MN)$; this identifies $\mathcal{X}(F_2, \SL_2(\C))$ with $\C^3$, recovering the Fricke-Vogt Theorem (for statement of this theorem see for example \cite{Fricke-Vogt}).  By incorporating an additional trace parameter $x_2 = \tr(MN^{-1})$, $\mathcal{X}(F_2, \SL_2(\C))$ can be identified with the hypersurface of solutions to $x_1 + x_2 - t_1t_2 = 0$ in $\C^4$. 

The tropical variety $\T\subset \Q^4$ of this hypersurface has a $2$ dimensional lineality space $L$ spanned by the vectors $(-1, -1, 0, -1)$ and $(-1, -1, -1, 0)$.  There are three maximal cones of $\T$, each obtained by adding an additional vector to the lineality space:

\begin{eqnarray*}
C_1 &=& \Q_{\geq 0}\{L, (-2, -2, 0, 0) \},~  \In_{C_1}(x_1 + x_2 - t_1t_2) = x_1 + x_2, \\
C_2 &=& \Q_{\geq 0}\{L, (0, 2, -1, -1) \},~  \In_{C_2}(x_1 + x_2 - t_1t_2) = x_1 - t_1t_2, \\
C_3 &=& \Q_{\geq 0}\{L, (2, 0, -1, -1) \},~  \In_{C_3}(x_1 + x_2 - t_1t_2) = x_2 - t_1t_2.\\
\end{eqnarray*}

Notably, each of the resulting initial forms of $x_1 + x_2 - t_1t_2$ is irreducible, it follows that $C_1, C_2,$ and $C_3$ are all prime cones.  Taking for a moment the prime cone $C_1$, Theorem \ref{th-intro-main2} implies that the weighting assigning $x_1, x_2, t_1, t_2$ the columns of the following matrix defines a rank $3$ valuation $\v_1: \C[\mathcal{X}(F_2, \SL_2(\C))] \to \Z^3$:

\begin{equation}
M_1 = \left[ \begin{array}{cccc}
-1  & -1 & -1 & 0 \\
-1  & -1 & 0 & -1 \\
-2 & -2 & 0 & 0 \end{array} \right]
\end{equation}

In \cite{M-NOK}, the second author describes a construction of a maximal rank valuation on the representation space $\mathcal{X}(F_g,G)$ for a free group of arbitrary rank and $G$ any connected complex reductive group.  For the $g = 2, G = \SL_2(\C)$ case, the necessary input of this construction is a trivalent graph $\Gamma$ with first Betti number $\beta_1$ equal to $2$, a choice of spanning tree in $\Gamma$, an orientation on the edges not in the chosen spanning tree, and a total ordering on the edges of $\Gamma$, see Figure \ref{3graph}.  The valuation $\v_1$ is constructed in \cite{M-NOK} as the maximal rank valuation associated to the leftmost graph in Figure \ref{3graph}, in particular each row corresponds to an edge in this graph, and the total ordering on edges can be interpreted as a total ordering on rows of $M_1$ (inducing a lexicographic ordering on standard monomials).  Matrices $M_2$ and $M_3$ can be constructed similarly for the cones $C_2$ and $C_3$, producing valuations $\v_2$ and $\v_3$; these valuations were constructed in \cite{M-NOK} in association with the middle and rightmost graphs in Figure \ref{3graph}. 

\begin{figure}[htbp]
\centering
\includegraphics[scale = 0.5]{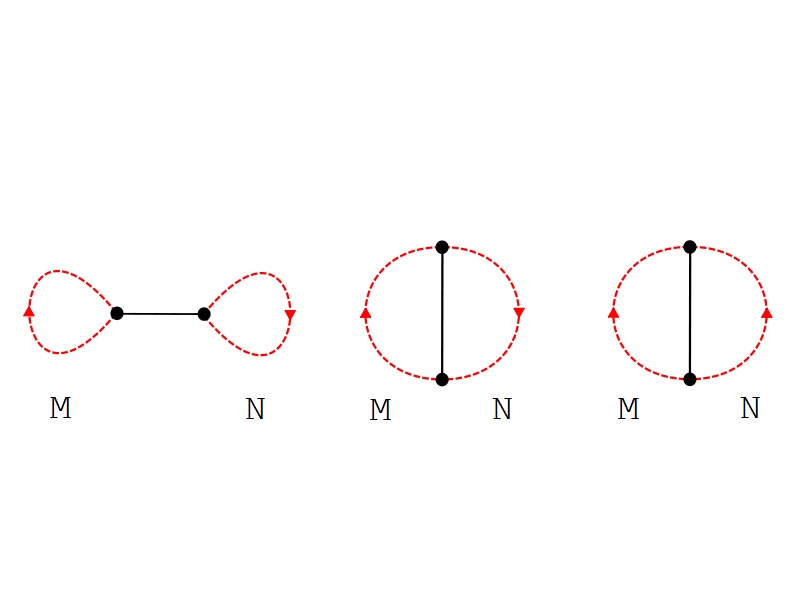}
\caption{Three trivalent graphs with $\beta_1 = 2$, chosen spanning trees (black), and oriented edges (red).}
\label{3graph}
\end{figure}

Notice that the non-spanning tree edges in each graph in Figure \ref{3graph} have been labeled with $M$ or $N$.  Each word $M, N, MN, MN^{-1}$ then corresponds to a unique closed cellular path through one of these graphs. The valuation $\v_1$ can then be derived by reading off the number of times each of these paths crosses an edge (see \cite{M-outer} for this construction).  Taking for example the leftmost graph, the word $MN^{-1}$  crosses the left loop $1$ time, the right loop $1$ time, and the middle edge $2$ times yielding the vector $(1, 1, 2)$; this is the negative of the first column of $M_1$. With the chosen orientation, the leftmost graph should not be able to distinguish between $MN$ and $MN^{-1}$, and indeed the second column of $M_1$ agrees with the first column.  In this way, the Newton-Okounkov polytopes constructed in \cite{M-NOK} for $\mathcal{X}(F_2, \SL_2(\C))$ can be reproduced using the tropical variety $\T$ and the matrices $M_i$. 
\end{example}

\begin{example}[Revisiting Example \ref{ex-Gobel-elliptic}]
\label{ex-Gobel-revisited}
The alternating invariants $A = \k[x_1, x_2, x_3]^{A_3}$ can be presented by the elementary symmetric polynomials $e_1 = x_1 + x_2 + x_3$, $e_2 = x_1x_2 + x_1x_3 + x_2x_3$, $e_3 = x_1x_2x_3$ along with the Vandermonde determinant form $y = (x_1 - x_2)(x_1-x_3)(x_2-x_3)$. The principal ideal which vanishes on these generators is generated by the polynomial $$f = e_1^2e_2^2 -4e_2^3 -4e_3e_1^3 + 18e_1e_2e_3 - 27e_3^2 - y^2.$$ The tropical variety $\T(f)$ contains three maximal prime cones corresponding to the initial ideals $\langle 4e_2^3 - y^2 \rangle, \langle4e_2^3  - 27e_3^2\rangle,$ and $\langle4e_3e_1^3 -y^2 \rangle$ respectively.  In particular, the generators $e_1, e_2, e_3, y$ are indeed a Khovanskii basis for the corresponding valuations. The valuations on $A$ built from these prime cones are not restrictions of initial term valuations on the polynomial algebra $\k[x_1, x_2, x_3]$ to $A$ as in SAGBI theory. This is confirmed by Example \ref{ex-Gobel-elliptic}(1).
\end{example}

\section{Appendix: Gr\"obner bases and higher rank tropical geometry}    \label{sec-groebner}
In this Appendix we introduce what we will need from the theory of Gr\"obner bases and tropical geometry. We extend the theory of monomial weightings to weightings by $\Q^r$, $r \geq 1$ equipped with a group ordering. This lays the groundwork for studying higher rank valuations and quasivaluations from an algorithmic perspective (see Section \ref{sec-weightandsubductive}).


\begin{remark} Higher rank versions of tropicalization have been studied by Foster and Ranganathan (\cite{FosterDhruv}).
\end{remark}  

\subsection{Gr\"obner theory} \label{subsec-Grobner-theory}
We refer the reader to \cite{St}, \cite{E}, and \cite{CLO} for the basics of the theory of Gr\"obner bases. We will consider weightings of the monomials in a polynomial ring $\k[\x] = \k[x_1, \ldots, x_n]$ by elements of $(\Q^r, \succ)$, where $\succ$ is a group ordering.  Although the geometric aspects of the term orders resulting from these weightings differ from those of weightings by $\Q$ (for example, we avoid discussion of the meaning of the  Gr\"obner fan when $r>1$), many of the algebraic and algorithmic properties of these term orders continue to hold.  We have shortened proofs when the essential ideas are already covered in $r =1$ case. 





\begin{definition}[Gr\"obner region]  \label{def-GR-r}
We define the higher rank Gr\"obner region $\textup{GR}^r(I) \subset \Q^{r \times n}$ of an ideal $I \subset \k[\x]$ as follows.
\begin{itemize}
\item[(1)] We say that $\w \in \Q^{r \times n}$ is in the {\it Gr\"obner region} $\textup{GR}^r(I)$ if and only if there is some monomial ordering $>$ such that the following holds:
\begin{equation}   \label{equ-GR}
\In_{>}(\In_\w(I)) = \In_{>}(I).
\end{equation}
For a fixed monomial ordering $>$, we denote the set of $\w$ satisfying \eqref{equ-GR} by $C^r_{>}(I)$.
\item[(2)] We also define the set $C_\w(I) \subset \Q^{r \times n}$ as the collection of those $\w' \in \Q^{r \times n}$ such that $\In_{\w'}(I) = \In_\w(I)$. 
\end{itemize}
\end{definition}

\begin{remark}
In the rank $1$ case, the points satisfying this definition of Gr\"obner region are in the Gr\"obner region as defined in the literature \cite[pg 13]{St}, \cite[Chapter 3]{JensenThesis}.  Given $u \in \Q^n$ and a monomial ordering $>$ such that $\In_>(\In_u(I)) = \In_>(I)$, one can show that $u$ is in the closure of the set of weights whose initial ideal agrees with the initial ideal of a weight from $\Q_{\geq 0}^n$: We can find $v \in \Q_{\geq 0}^n$ and a generating set $k_1, \ldots, k_m \in I$ such that the $\In_>(k_i)$ generate $\In_>(I)$ and $\In_v(k_i) = \In_v(\In_u(k_i) = \In_>(k_i)$ for each $1 \leq i \leq m$. It follows that $\In_v(I) = \In_{u + \epsilon v}(I)$ for small $\epsilon$, so $u$ is in the closure of the set of weights whose initial ideal agrees with $\In_v(I)$. 
\end{remark}

The monomials not contained in $\In_{>}(I)$ are usually called {\it standard monomials}. It is well-known that the images of standard monomials in $\k[\x]/I$ are a vector space basis for this quotient which we denote by $\mathbb{B}_>(I)$ (or simply by $\mathbb{B}$ when there is no chance of confusion). The proof of the following lemma is exactly as the proof of \cite[Lemma 3.1.11]{JensenThesis}.

\begin{lemma}   \label{initialformwrite}
Any $h \in \In_\w(I)$ can be written as a sum $\sum_i \In_\w(f_i)$ for $f_i \in I$, where the summands all have different homogeneous $\w$-degrees. 
\end{lemma}

Let $G_>(I)$ denote the reduced Gr\"obner basis of $I$ with respect to $>$. As with the standard Gr\"obner theory, it is possible to check for membership in $C^r_{>}(I)$ using $G_>(I)$.  


\begin{lemma}\label{Cinequalities}
A weight $\w \in \Q^{r \times n}$ is in $C^r_{>}(I)$ if and only if $\In_>(\In_\w(g)) = \In_>(g)$ for all $g \in G_>(I)$. In particular, $C^r_{>}(I)$ is defined by a finite set of inequalities.  Furthermore, 

\begin{equation}
G_>(\In_\w(I)) = \{\In_\w(g) \mid g \in G_>(I)\},
\end{equation}

{is the reduced Gr\"obner basis for $\In_\w(I)$ with respect to $>$.}


\end{lemma}
\begin{proof}
This is a standard result in Gr\"obner theory, see \cite[Lemma 3.1.12]{JensenThesis}.  
\end{proof}

The next lemma gives a characterization of the set $C_\w(I)$ for when 
$\w $ lies in the Gr\"obner region $\GR^r(I)$. The proof is exactly as in the proof of \cite[Proposition 3.1.4]{JensenThesis}.
\begin{lemma}\label{Cgeninequalities}
Let $\w \in C^r_{>}(I)$, then $\w' \in C_\w(I)$ if and only if $\In_\w(g) = \In_{\w'}(g)$ for all $g \in G_{>}(I)$.
\end{lemma}


The negative orthant is always part of the Gr\"obner region when $r = 1$ (recall that we are using the MIN convention); we give a generalization of this fact to $r \geq 1$. The Hahn embedding theorem (see Section \ref{subsec-valuation}) implies that there is an embedding of ordered groups $\eta: \Q^r \to \R^r$, where $\R^r$ is given the standard lexicographic ordering. In particular, the subset $Q^- = \eta^{-1}(\eta(\Q^r) \cap (\R_{\leq 0})^r)$ has the property that for any lattice $L \subset \Q^r$, the set $L \cap Q^-$ is maximum well-ordered. There is always an element $\textbf{1} \in Q^-$ such that for any $w \in \Q^r$, we have $w + N \textbf{1} \in Q^-$ for $N$ sufficiently large.  If $\succ$ is taken to be the standard lexicographic ordering, then $Q^- = (\Q_{\leq 0})^r$ and $\textbf{1} = (-1, \ldots, -1) \in \Q^r$.  

\begin{lemma}   \label{lem-negativeGregion}
For any $I \subset \k[\x]$ we have  $(Q^-)^n \subset \textup{GR}(I)$.  Furthermore, if $I$ is homogeneous with respect to a positive grading then $\textup{GR}^r(I) = \Q^{r \times n}$.
\end{lemma}
\begin{proof}
If $\w \in (Q^-)^n$ then for any monomial $\x^{\alpha}$ there are only finitely many monomials $\x^{\beta}$
with $M \beta \succ M \alpha$.  We can define the composite ordering $>_\w$ as in \cite[Proposition 1.8]{St} and conclude that it refines the monomial weighting by $M$.  Now, if $I$ is homogeneous with respect to $(d_1, \ldots, d_n) \in \Q_{\geq 0}^n$ we let $D \in (Q^-)^n$ be the matrix $[d_1 \textbf{1}, \ldots, d_n\textbf{1}]$;  then any multiple of $D$ can be added to $\w$ without altering the initial ideal.
\end{proof}


\subsection{Comparison of initial ideals}
Let $\Q^r$ be equipped with the standard lexicographic ordering $\succ$.  We show that the initial ideals obtainable by $\w \in \Q^{r \times n}$ are the same as those in the $r =1 $ case; this statement is used in Section \ref{sec-val-from-prime-cone}.

\begin{lemma}\label{lem-in_w-in_u_i}
Let $\w \in \Q^{r \times n}$ and let $\u_1, \ldots, \u_r$ be the rows of the matrix $\w$. 
Then the following initial ideals coincide:
\begin{equation}
\In_\w(I) = \In_{\u_r}(\ldots \In_{\u_1}(I) \ldots ).
\end{equation}
\end{lemma}
\begin{proof}
First observe that for any $f \in I$ we have $\In_\w(f) =  \In_{\u_r}(\ldots \In_{\u_1}(f)\ldots )$. This implies that $\In_\w(I) \subset \In_{\u_r}(\ldots \In_{\u_1}(I)\ldots )$. To prove the other inclusion, we proceed by induction on $r$. Let $h \in \In_{\u_r}(\ldots \In_{\u_1}(I)\ldots )$. We can assume that $\In_{\u_{r-1}}(\ldots \In_{\u_1}(I)\ldots ) = \In_{\w'}(I)$, where $\w' \in \Q^{(r-1) \times n}$ is the matrix with rows $\u_1, \ldots, \u_{r-1}$, regarded as a $\Q^{r-1}$-monomial weighting.  Use Lemma \ref{initialformwrite} to write $h = \sum_i \In_{\u_r}(f_i)$ where $f_i \in \In_{\w'}(I)$ and each $\In_{\u_r}(f_i)$ has a distinct homogeneous $\u_r$-degree. This implies that no monomials are shared among the $\In_{\u_r}(f_i)$.  Another application of Lemma \ref{initialformwrite} implies that each $f_i$ can be written as $\sum_j \In_{\w'}(g_{ij})$ for $g_{ij} \in I$, where each $\In_{\w'}(g_{ij})$ has a distinct homogeneous $\w'$-degree. It follows that $\In_{\u_r}(f_i) = \sum_i \In_{\u_r}(\In_{\w'}(g_{ij}))$ and hence $h = \sum_{i,j} \In_{\u_r}(\In_{\w'}(g_{ij}))$.
\end{proof}

For any $\u \in \Q^n$ it is clear that we can find $\w \in \Q^{r \times n}$ such that $\In_{\u}(I) = \In_{\w}(I)$. Now we show that it is possible, given $\w \in \Q^{r \times n}$, to find $\u \in \Q^n$ with the same initial ideal as $\w$. Let $A_i = \{\alpha_{i1}, \ldots, \alpha_{in_i} \} \subset \Q^r$, $1 \leq i \leq m$, be finite sets (later we will take $A_i$ to be the set of $\w$-weights of monomials in some polynomial $f_i$ for some weighting matrix $\w \in \Q^{r \times n}$). We denote the smallest element in each $A_i$ by $\beta_i$.  Also, for $\u, v \in \Q^r$, we write $\u \cdot v$ for the standard inner product. The next lemma is a consequence of \cite[Proposition 1.11]{St}. 

\begin{lemma}\label{lem-weight-approx}
There is a vector $v \in \Q_{\geq 0}^r$ such that $v \cdot (\beta_i -  \alpha_{ij}) < 0$ whenever $\alpha_{ij} \succ \beta_i$.  
\end{lemma}


Given a finite number of polynomials, Lemma \ref{lem-weight-approx} shows that for every $\w \in \Q^{r \times n}$ there exists $\u \in \Q^n$ such that the initial forms of these polynomials with respect to $\w$ are the same as their initial forms with respect to $\u$. 

\begin{proposition}  \label{prop-in-u-rep-in-w}
Let $I \subset \k[\x]$ be an ideal, then for any $\w \in \Q^{r \times n}$ there is some $\u \in \Q^n$ such that $\In_{\w}(I) = \In_{\u}(I)$. 
\end{proposition}
\begin{proof}
First we assume that $I$ is homogeneous with respect to a positive grading on $\k[\x]$. This implies that $\w \in C^r_>(I) \subset \textup{GR}^r(I)$ for some monomial ordering $>$. Let $G_>(I) \subset I$ be the associated reduced Gr\"obner basis. By Lemma \ref{lem-weight-approx}, we can find $v \in \Q_{\geq 0}^r$ such that if $\u = v^T M$ then we have $\In_{\w}(g) = \In_{\u}(g)$, for all $g \in G_>(I)$, so that $\In_{\u}(I) = \In_{\w}(I)$.

Next, we let $I$ be a general ideal, i.e. not necessarily homogeneous. We then form the homogenization $I_h \subset \k[x_0, \x]$ (see \cite[Chapter 2]{MSt}). Let $(0, \w) \in \Q^{r \times (n+1)}$ be the matrix obtained from $\w$ by adding a $0$ column to the left. The proof of \cite[Proposition 2.6.1]{MSt} shows that for any $\w \in \Q^{r \times n}$, we have $\In_{(0, \w)}(I_h)_{x_0 = 1} = \In_{\w}(I)$, where $\In_{(0, \w)}(I_h)_{x_0 = 1} \subset \k[\x]$ is the ideal obtained by setting $x_0$ equal to $1$. Since $x_0$ is weighted $0$, an application of Lemma \ref{lem-weight-approx} to an appropriate Gr\"obner basis of $I_h$ produces a vector $(0, \u) \in \Q^{n+1}$.  Now we observe that
$$\In_{\w}(I) = \In_{(0, \w)}(I_h)_{x_0 = 1} = \In_{(0, \u)}(I_h)_{x_0 = 1} = \In_{\u}(I).$$
\end{proof}

It is easy to find a weighting $\u \in \Q^n$ as in Proposition \ref{prop-in-u-rep-in-w} if the rows of the weighting matrix $\w$ are taken from the same cone in the Gr\"obner fan $\Sigma(I_h)$ (or $\Sigma(I)$ if $I$ is itself homogeneous with respect to a positive grading).
\begin{proposition}\label{midpoint}    
Suppose that the rows of the weighting matrix $\w$ are linearly independent, are taken from the same cone $C \subset \Sigma(I_h)$ and moreover span the linear span of $C$. Then $\In_{\w}(I) = \In_{\u}(I)$ for $\u = \sum_i \u_i$. 
\end{proposition}
\begin{proof}
From Lemma \ref{lem-in_w-in_u_i} and \cite[Proposition 1.13]{St}, we can conclude that there are $\epsilon_2, \ldots, \epsilon_r > 0$ such that $\In_M(I) = \In_{u_1+\epsilon_2 u_2 + \cdots + \epsilon_r u_r}(I)$ which in turn is equal to $\In_{u_1+\cdots+u_r}(I)$  
\end{proof}

Finally we end the section with the definition of lineality space of an ideal. 
\begin{definition}[Lineality space]   \label{def-lineality-space}
The {\it lineality space} $L^r(I)$ of an ideal $I \subset \k[\x]$ is the set of all $\w \in \Q^{r\times n}$ with $\In_\w(I) = I$. \end{definition}


\subsection{Tropical Geometry}\label{tropical}
We briefly recall the notion of tropical variety of an ideal, and an extension of this notion to higher ranks. We suggest the book by Maclagan and Sturmfels \cite{MSt} for an excellent introduction to tropical geometry.  We will confine ourselves to tropicalization over a trivially valued field $\k$. 

\begin{definition}   \label{def-trop-var}
Let $I \subset \k[\x]$ be an ideal. The tropical variety $\T(I) \subset \Q^n$ is the set of $\u \in \Q^n$ such that $\In_\u(I)$ contains no monomials. 
\end{definition}

The tropical variety carries the structure of a weighted polyhedral fan \cite{MSt}. Furthermore, if $I$ is homogeneous with respect to some positive grading on $\k[\x]$, it can be considered to be a subfan of the Gr\"obner fan $\Sigma(I)$ (\cite{SpSt}).  More generally, if $I$ is not homogeneous, or is an ideal in a Laurent polynomial ring, the tropical variety can be studied through its relationship with the homogenization $I_h$. In particular, the tropical variety $\T(I)$ can be taken to be a union of faces of the Gr\"obner fan of $I_h$ intersected with the hyperplane defined by setting the $x_0$-weight equal to $0$ (\cite[Proposition 2.6.2]{MSt}). A consequence of this construction is that there is a subdivision of $\T(I)$ into open polyhedral cones $C$ such that if $\u, \u'$ are in the same cone, then $\In_\u(I) = \In_{\u'}(I)$ (\cite[proof of Theorem 2.6.5]{MSt}). 

The following is clear from the definition of $\T(I)$.
\begin{lemma}
If $\In_\u(I)$ is prime and $\{x_1, \ldots, x_n\} \cap \In_\u(I) = \emptyset$ then $\u \in \T(I)$. 
\end{lemma}

The initial ideals $\In_\u(I)$ for $\u \in \T(I)$ share some of the properties of $I$. For example, for an arbitrary $\u \in \Q^n$, the initial ideal $\In_\u(I)$ could be all of $\k[\x]$, however if $\u \in \T(I)$ then the dimension of  $\k[\x]/\In_\u(I)$ is equal to that of $\k[\x]/I$ (\cite[Theorem 8.2.1]{JensenThesis}).


\begin{definition}  \label{def-trop-var-higher-rank}
Let $I \subset \k[\x]$ be an ideal. Consider $\Q^r$ equipped with a group ordering. We say that $\w \in \Q^{r \times n}$ is in the {\it rank $r$ tropical variety} $\T^r(I)$ if the initial ideal $\In_{\w}(I)$ contains no monomials. 
\end{definition}

The points in the rank $r$ tropical variety $\T^r(I)$ are related to the points in the usual tropical variety $\T(I)$ by the following proposition. It is a corollary of Lemma \ref{lem-in_w-in_u_i}.
\begin{proposition}\label{tropicaltower}
Let $\Q^r$ be equipped with the standard lexicographic ordering. Let $\w \in \Q^{r \times n}$ and let $\u_i \in \Q^n$ be the $i$-th row of $M$, $1 \leq i \leq r$. Then $\w \in \T^r(I)$ if and only if $u_1 \in \T(I)$ and $\u_i \in \T(\In_{\u_{i-1}}(\ldots \In_{\u_1}(I)\ldots ))$ for all $1 < i \leq r$.
\end{proposition}

\bibliographystyle{alpha}
\bibliography{Biblio}
\end{document}